\documentclass[article,12pt]{amsart}

\usepackage{amsfonts}
\usepackage{amsmath, amssymb}
\usepackage{mathrsfs}
\usepackage{graphicx}
\usepackage{color}
\usepackage{epsfig}
\usepackage[font={small}]{caption}
\usepackage{xcolor}

\numberwithin{equation}{section}
\theoremstyle{plain}
\newtheorem{thm}{Theorem}[section]
\newtheorem{rem}{Remark}[section]

\newtheorem{lem}{Lemma}[section]

\newcommand{\tn}{\theta_{n}}
\newcommand{\tnp}{\theta_{n+1}}
\newcommand{\ta}{\theta_{\alpha}}
\newcommand{\vta}{\vartheta_{\alpha}}

\newcommand{\vtnh}{\widehat{\vartheta}_{n}}
\newcommand{\vtkh}{\widehat{\vartheta}_{k}}
\newcommand{\vtnt}{\widetilde{\vartheta}_{n}}
\newcommand{\vtkt}{\widetilde{\vartheta}_{k}}

\newcommand{\vtnph}{\widehat{\vartheta}_{n+1}}
\newcommand{\vtnpt}{\widetilde{\vartheta}_{n+1}}

\newcommand{\dE}{\mathbb{E}}
\newcommand{\dR}{\mathbb{R}}

\newcommand{\dP}{\mathbb{P}}

\newcommand{\cE}{\mathcal{E}}
\newcommand{\cG}{\mathcal{G}}

\newcommand{\cM}{\mathcal{M}}
\newcommand{\cH}{\mathcal{H}}
\newcommand{\cN}{\mathcal{N}}

\newcommand{\cV}{\mathcal{V}}
\newcommand{\cW}{\mathcal{W}}
\newcommand{\cX}{\mathcal{X}}
\newcommand{\cY}{\mathcal{Y}}
\newcommand{\rI}{\mathrm{I}}
\newcommand{\cF}{\mathcal{F}}

\newcommand{\cR}{\mathcal{R}}

\newcommand{\veps}{\varepsilon}

\newcommand{\ccMc}{\langle \cM \rangle}

\newcommand{\wh}{\widehat}

\newcommand{\ind}{\mbox{1}\kern-.25em \mbox{I}}
\font\calcal=cmsy10 scaled\magstep1

\def\build#1_#2^#3{\mathrel{\mathop{\kern 0pt#1}\limits_{#2}^{#3}}}
\def\liml{\build{\longrightarrow}_{}^{{\mbox{\calcal L}}}}

\def\videbox{\mathbin{\vbox{\hrule\hbox{\vrule height1.4ex \kern.6em\vrule height1.4ex}\hrule}}}
\def\demend{\hfill $\videbox$\\}

\begin{document}
\title[Stochastic algorithms for superquantiles estimation]
{Stochastic approximation algorithms for superquantiles estimation}
\author{Bernard Bercu}
\thanks{The corresponding author is Bernard Bercu, email address: bernard.bercu@math.u-bordeaux.fr}
\address{Institut de Math\'ematiques de Bordeaux, Universit\'e de Bordeaux, 
UMR 5251, 351 Cours de la Lib\'eration, 33405 Talence cedex, France.}
\email{bernard.bercu@math.u-bordeaux.fr}
\author{ Manon Costa}
\address{Institut de Math\'ematiques de Toulouse,  Universit\'e de Toulouse,
UMR 5219, 118 Route de Narbonne, 31062 Toulouse cedex, France.}
\email{manon.costa@math.univ-toulouse.fr}
\author{ S\'ebastien Gadat}
\address{Toulouse School of Economics, Universit\'e de Toulouse, 
UMR 5604 and Institut Universitaire de France}
\email{sebastien.gadat@tse-fr.eu}
%\date{\today}

%%%%%%%%%%%%%%%%%%%%%%%%%%%%%%%%%%%%%%%%%%%%%%%%%%%%%%%%%%%%%%%%%%%%%%%%%%%%%%%%%%%%%%%%%%%%%%%%%%

\begin{abstract}
This paper is devoted to two different two-time-scale stochastic approximation algorithms for superquantile estimation.
We shall investigate the asymptotic behavior of a Robbins-Monro estimator and its convexified version.
Our main contribution is to establish the almost sure convergence, the quadratic strong
law and the law of iterated logarithm for our estimates via a martingale approach. 
A joint asymptotic normality is also provided. Our theoretical analysis is illustrated by numerical experiments on
real datasets.
\end{abstract}


\maketitle

\vspace{-5ex}
%%%%%%%%%%%%%%%%%%%%%%%%%%%%%%%%%%%%%%%%%%%%%%%%%%%%%%%%%%%%%%%%%%%%%%%%%%%%%%%%%%%%%%%%%%%%%%%%%%

\section{Introduction}
\label{S-I}

%%%%%%%%%%%%%%%%%%%%%%%%%%%%%%%%%%%%%%%%%%%%%%%%%%%%%%%%%%%%%%%%%%%%%%%%%%%%%%%%%%%%%%%%%%%%%%%%%%

Estimating quantiles has a longstanding history in statistics and probability. Except in parametric models where explicit formula are available, 
the estimation of quantiles is a real issue. The most commun way to estimate quantiles is to make use of order statistics, see
among other references \cite{Bahadur1966,Ghosh1971}.
Another strategy is to make use of stochastic approximation algorithms and the pioneering work in this vein is the celebrated paper
by Robbins and Monro \cite{RobbinsMonro1951}. 

\vspace{1ex}
\noindent
Let $X$ be an integrable continuous random variable with strictly increasing cumulative distribution function $F$ and probability density function $f$.
For any $\alpha \in ]0,1[$, the quantile $\ta$ of order $\alpha$ of $F$ is given by
\begin{equation}
\label{DEFQ}
F(\ta)= \dP(X \leq \ta)=\alpha,
\end{equation}
whereas the superquantile $\vta$ of order $\alpha$ is defined by
\begin{equation}
\label{DEFSQ}
\vta = \dE[ X \, \vert \, X \geq \ta] =\frac{\dE[X \rI_{\{X \geq \ta\}}]}{\dP(X \geq \ta)} = \frac{\dE[X \rI_{\{X \geq \ta\}}]}{1-\alpha}.
\end{equation}
One can observe that the superquantile provides more information on the tail of the distribution of the random variable $X$. 
Our goal in this paper is to simultaneously estimate quantiles and superquantiles, 
also respectively known as values at risk and conditional values at risk, which have become increasingly popular as measures of risk in finance \cite{Rockafellar2000,Rockafellar2002}.

\vspace{1ex}
\noindent
The paper is organized as follows. Section \ref{S-O} is devoted to a brief overview of the previous literature on the recursive estimation of
quantiles and superquantiles. The main results of the paper are given in Section \ref{S-MR}. 
We propose the almost sure convergence of two-time-scale stochastic approximation algorithms for superquantile estimation. 
The quadratic strong law (QSL) as well as the law of iterated logarithm (LIL) of our stochastic algorithms are also provided.
Moreover, we establish the joint asymptotic normality of our estimates. Numerical experiments on real data are given in Section \ref{S-NE}. 
All technical proofs are postponed to Appendices A and B.

%%%%%%%%%%%%%%%%%%%%%%%%%%%%%%%%%%%%%%%%%%%%%%%%%%%%%%%%%%%%%%%%%%%%%%%%%%%%%%%%%%%%%%%%%%%%%%%%%%

\section{Overview of existing literature}
\label{S-O}

%%%%%%%%%%%%%%%%%%%%%%%%%%%%%%%%%%%%%%%%%%%%%%%%%%%%%%%%%%%%%%%%%%%%%%%%%%%%%%%%%%%%%%%%%%%%%%%%%%

A wide range of literature exists already on the recursive estimation of quantiles \cite{RobbinsMonro1951}. However, to the best of our knowledge,
only a single paper is available on the recursive estimation of superquantiles \cite{Bardou2009}.
In many practical situations where the data are recorded online with relatively high speed, or when the data
are simply too numerous to be handled in batch systems, it is more suitable to implement a recursive strategy where 
quantiles and superquantiles are sequentially estimated with the help of stochastic approximation algorithms \cite{Duflo1997},
\cite{KushnerYin2003}. We also refer the reader to \cite{CCG2017, CCZ2013, Godichon2015, Godichon2019} for the online estimation of geometric medians
and variances.

\vspace{1ex}
\noindent
Bardou et al. \cite{Bardou2009} have previously studied the averaged version \cite{PolyakJuditsky1992, Ruppert1988}
of a one-time-scale stochastic algorithm in order to estimate $\ta$ and $\vta$. Here, we have chosen to investigate a two-time-scale stochastic 
algorithm \cite{Borkar1997, GPS2018, Konda2004, MokkademPelletier2006} which performs pretty well and offers more flexibility than the one-time-scale algorithm.
Let $(X_n)$ be a sequence of independent and identically distributed random variables sharing the same distribution as $X$.
We shall extend the statistical analysis of \cite{Bardou2009} by studying the two-time-scale stochastic 
algorithm given, for all $n \geq 1$, by
\begin{equation}
\label{TTSALGO1}
\left \{
\begin{aligned}
&\tnp =\tn-a_{n} \Bigl( \rI_{\{X_{n+1} \leq \tn \}} - \alpha \Bigr), \\
&\vtnph=\vtnh+b_{n} \Bigl( \frac{X_{n+1}}{1-\alpha}\rI_{\{X_{n+1} >\tn\}} - \vtnh \Bigr),
\end{aligned}
\right.
\end{equation} 
where the initial values $\theta_1$ and $\wh{\vartheta}_1$ are square integrable random variables
which can be arbitrarily chosen and the steps $(a_n)$ and $(b_n)$ are two positive sequences of 
real numbers strictly smaller than one, decreasing towards zero such that
\begin{equation}
\label{CONDSTEP}
\sum_{n=1}^\infty a_n=+\infty, \hspace{0.4cm}
\sum_{n=1}^\infty b_n=+\infty
\hspace{0.8cm}\text{and}\hspace{0.8cm}
\sum_{n=1}^\infty a_n^2<+\infty, \hspace{0.4cm}
\sum_{n=1}^\infty b_n^2<+\infty.
\end{equation}
We shall also investigate the asymptotic behavior of the convexified version of  algorithm \eqref{TTSALGO1},
based on the Rockafellar-Uryasev's identity \cite{Rockafellar2000} and
given, for all $n \geq 1$, by
\begin{equation}
\label{TTSALGO2}
\left \{
\begin{aligned}
&\tnp =\tn-a_n \Bigl( \rI_{\{X_{n+1} \leq \tn\}}-\alpha\Bigr)\\
&\vtnpt=\vtnt+b_n \Bigl(\tn+ \frac{(X_{n+1}-\tn)}{1-\alpha}\rI_{\{X_{n+1} >\tn\}} - \vtnt\Bigr),
\end{aligned}
\right.\end{equation}
where as before the initial values $\theta_1$ and $\widetilde{\vartheta}_1$ are square integrable random variables
which can be arbitrarily chosen. We also refer the reader to the original contribution \cite{Ben-Tal} 
where this convexification first appeared.
The almost sure convergence 
\begin{equation}
\label{ASCVGRM}
\lim_{n \rightarrow \infty}\tn=\ta \hspace{1cm}\text{a.s.}
\end{equation}
is a famous result that was established by Robbins and Monro \cite{RobbinsMonro1951},
Robbins and Siegmund \cite{RobbinsSiegmund1971}. Moreover, the
asymptotic normality is due to Sacks, see Theorem 1 in \cite{Sacks1958}. It requires the
additional assumption that the probability density function $f$ is differentiable with bounded derivative in
every neighborhood of the quantile $\ta$. 
More precisely, if the step $a_n=a_1/n$ where $a_1>0$ and $2 a_1 f(\ta)>1$, we have the asymptotic normality
\begin{equation}
\label{ANRM}
\sqrt{n}\bigl(\tn - \ta \bigr) \liml \cN\Bigl(0, \frac{a_1^2\alpha(1-\alpha)}{2a_1f(\ta) -1} \Bigr).
\end{equation}
One can observe that in the special case where the value $f(\ta)>0$ is known, 
it is possible to minimise the previous limiting variance by choosing $a_1=1/f(\ta)$ and to
obtain from \eqref{ANRM} the asymptotic efficiency
$$
\sqrt{n}\bigl(\tn - \ta \bigr) \liml \cN\Bigl(0, \frac{\alpha(1-\alpha)}{f^2(\ta)} \Bigr).
$$
Some useful refinements on the asymptotic behavior of the sequence $(\tn)$ are also well-known.
The LIL was first proved by Gaposhkin and Krasulina, see Theorem 1 in \cite{Gaposkin1975} 
and Corollary 1 in \cite{Kersting1977}. More precisely, if the step $a_n=a_1/n$ where $2a_1 f(\ta)>1$, 
we have the LIL
\begin{eqnarray}  
\limsup_{n \rightarrow \infty} \left(\frac{n}{2 \log \log n} \right)^{1/2}
 \!\!\!\bigl( \tn - \ta \bigr)
&=& - \liminf_{n \rightarrow \infty}
\left(\frac{n}{2 \log\log n}\right)^{1/2} 
\!\!\!\bigl( \tn - \ta \bigr) \notag \\
&=& \left(  \frac{a_1^2\alpha(1-\alpha)}{2a_1f(\ta) -1}  \right)^{1/2}
\hspace{1cm}\text{a.s.}
\label{LILRM}
\end{eqnarray}
In particular, it follows from \eqref{LILRM} that
\begin{equation}
\label{LILSUPRM}  
\limsup_{n \rightarrow \infty} \left(\frac{n}{2 \log \log n} \right)
\bigl( \tn - \ta \bigr)^2= \frac{a_1^2\alpha(1-\alpha)}{2a_1f(\ta) -1}
\hspace{1cm}\text{a.s.}
\end{equation} 
which is the limiting variance in \eqref{ANRM}.
The QSL is due to Lai and Robbins, see Lemma 1 and Theorem 2 in \cite{LaiRobbins1979}
as well as Theorem 3 in \cite{Pelletier1998}. More precisely, they proved that
\begin{equation}
\label{QSLRM}
\lim_{n \rightarrow \infty} \frac{1}{\log n} \sum_{k=1}^n \bigl( \theta_k - \ta \bigr)^2= \frac{a_1^2\alpha(1-\alpha)}{2a_1f(\ta) -1}
\hspace{1cm}\text{a.s.}
\end{equation}
Besides the classical choice $a_n=a_1/n$ where $a_1>0$, slower step-size $a_n=a_1/n^a$ where $a_1>0$ and $1/2<a<1$ have been studied in depth.
We refer the reader to pioneer work of Chung \cite{Chung1954} and to Fabian \cite{Fabian1968} who obtained that the asymptotic normality 
still holds for the Robbins-Monro algorithm. More precisely, if $f(\ta)>0$, they showed that
\begin{equation}
\label{ANRMStepa}
\sqrt{n^a}\bigl(\tn - \ta \bigr) \liml \cN\Bigl(0, \frac{a_1\alpha(1-\alpha)}{2f(\ta)} \Bigr).
\end{equation}
In addition, it follows from Lai and Robbins \cite{LaiRobbins1979} or Pelletier \cite{Pelletier1998} that
\begin{eqnarray}  
\limsup_{n \rightarrow \infty} \left(\frac{n^a}{2(1-a) \log n} \right)^{1/2}
 \!\!\!\bigl( \tn - \ta \bigr)
&=& - \liminf_{n \rightarrow \infty}
\left(\frac{n^a}{2(1-a) \log n}\right)^{1/2} 
\!\!\!\bigl( \tn - \ta \bigr) \notag \\
&=& \left(  \frac{a_1\alpha(1-\alpha)}{2f(\ta)}  \right)^{1/2}
\hspace{1cm}\text{a.s.}
\label{LILRMStepa}
\end{eqnarray}
In particular, 
\begin{equation}
\label{LILSUPRMStepa}  
\limsup_{n \rightarrow \infty} \left(\frac{n^a}{2(1-a) \log n} \right)
\bigl( \tn - \ta \bigr)^2= \frac{a_1\alpha(1-\alpha)}{2f(\ta)}
\hspace{1cm}\text{a.s.}
\end{equation} 
Moreover, we also have from \cite{LaiRobbins1979}, \cite{Pelletier1998} that
\begin{equation}
\label{QSLRMStepa}
\lim_{n \rightarrow \infty} \frac{1}{n^{1-a}} \sum_{k=1}^n \bigl( \theta_k - \ta \bigr)^2= \frac{a_1\alpha(1-\alpha)}{2(1-a)f(\ta)}
\hspace{1cm}\text{a.s.}
\end{equation}
The restrictive assumption $2a_1 f(\ta)>1$, which involves the knowledge of $f(\ta)$, is no longer needed. However, the convergence rate
$n^a$ is always slower than $n$, which means that the choice $a_n=a_1/n$ theoretically outperforms the one of $a_n=a_1/n^a$, 
at least asymptotically.

\vspace{1ex}
\noindent
In the special case of the one-time-scale stochastic algorithm where $a_n=b_n$, Bardou et al. \cite{Bardou2009} 
proved the almost sure convergences
\begin{equation}
\label{ASCVGSQB1}
\lim_{n \rightarrow} \tn= \ta \hspace{1.5cm}\text{and}\hspace{1.5cm}
\lim_{n \rightarrow} \vtnt=\vta
 \hspace{1cm}\text{a.s.}
\end{equation}
using an extended version of Robbins-Monro theorem together with Cesaro and Kronecker lemmas, see e.g. Theorem 1.4.26 in \cite{Duflo1997}.
They also state without proof that
\begin{equation}
\label{ASCVGSQB2}
\lim_{n \rightarrow} \vtnh=\vta
 \hspace{1cm}\text{a.s.}
\end{equation}
Yet, other almost sure asymptotic properties for the sequences $(\vtnh)$ and $(\vtnt)$, such as the LIL and the 
QSL, are still missing. Bardou et al. also established in Theorem 2.4 of \cite{Bardou2009} the joint asymptotic normality
of the averaged version \cite{PolyakJuditsky1992, Ruppert1988} of their one-time-scale stochastic algorithm
\begin{equation}
\label{ANRMB}
\sqrt{n}
\begin{pmatrix}
\overline{\theta}_n - \ta \\
\overline{\vartheta}_n - \ta \\
\end{pmatrix} \liml \cN\bigl(0, \Sigma \bigr)
\end{equation}
where the asymptotic covariance matrix $\Sigma$ is explicitly calculated, 
$$
\overline{\theta}_n = \frac{1}{n}\sum_{k=1}^n \theta_k \hspace{1.5cm}\text{and}\hspace{1.5cm}
\overline{\vartheta}_n   = \frac{1}{n}\sum_{k=1}^n \widetilde{\vartheta}_{k}.
$$
We will show that our two-time-scale stochastic algorithms given by \eqref{TTSALGO1} and \eqref{TTSALGO2} allow us to avoid the 
Ruppert and Polyak-Juditsky averaging principle. Moreover, they perform pretty well both from a theoretical and a practical point of view
and offer more flexibility than the one-time-scale stochastic algorithm.

%%%%%%%%%%%%%%%%%%%%%%%%%%%%%%%%%%%%%%%%%%%%%%%%%%%%%%%%%%%%%%%%%%%%%%%%%%%%%%%%%%%%%%%%%%%%%%%%%%

\vspace{-1ex}
\section{Main results}
\label{S-MR}

%%%%%%%%%%%%%%%%%%%%%%%%%%%%%%%%%%%%%%%%%%%%%%%%%%%%%%%%%%%%%%%%%%%%%%%%%%%%%%%%%%%%%%%%%%%%%%%%%%

In order to state our main results, it is necessary to introduce some assumptions.
\begin{displaymath}
\begin{array}{ll}
(\mathcal{A}_1) & \textrm{The probability density function $f$ is differentiable with bounded derivative in}\\ 
  & \textrm{every neighborhood of $\ta$.}
\end{array}
\end{displaymath}
\vspace{-1ex}
\begin{displaymath}
\begin{array}{ll}
(\mathcal{A}_2) & \textrm{The function $\Phi$ defined, for all $\theta \in \dR$, by  $\Phi(\theta)=f(\theta)+ \theta f^{\prime}(\theta)$ is bounded in}\\ 
  & \textrm{every neighborhood of $\ta$.}
\end{array}
\end{displaymath}

Our first result concerns the basic almost sure convergence of the two-time-scale stochastic algorithms \eqref{TTSALGO1} and \eqref{TTSALGO2} 
to the superquantile $\vta$.

\begin{thm}
\label{T-ASCVGSQ} Assume that $(\mathcal{A}_1)$ holds and that the random variable $X$ is square integrable.
Then, we have the almost sure convergences
\begin{equation}
\label{ASCVGSQ1}
\lim_{n \rightarrow \infty}\vtnh=\vta \hspace{1cm}\text{a.s.}
\end{equation}
\begin{equation}
\label{ASCVGSQ2}
\lim_{n \rightarrow \infty}\vtnt=\vta \hspace{1cm}\text{a.s.}
\end{equation}
\end{thm}

\noindent
Our proof is slightly different from that of Bardou et al. \cite{Bardou2009} established for the one-time-scale stochastic algorithm 
where $a_n=b_n$. It can be found in Appendix A for sake of completeness.
We now focus our attention on the almost sure rates of convergence of the sequences $(\vtnh)$ and $(\vtnt)$.
We divide our analysis into two parts depending on the step size $(b_n)$ in the superquantile recursive procedure. First of all, we 
shall consider the optimal step  $b_n=b_1/n$. Then, we shall study the case where $b_n=b_1/n^{b}$ with $1/2<b<1$. For all
$\theta \in \dR$, denote
\begin{equation}
\label{DEFVAR}
\sigma_\alpha^2(\theta)=\frac{1}{(1-\alpha)^2} \text{Var}(X \rI_{\{X >\theta\}}) \hspace{0.5cm}\text{and}\hspace{0.5cm}
\tau_\alpha^2(\theta)=\frac{1}{(1-\alpha)^2} \text{Var}((X-\theta) \rI_{\{X >\theta\}}).
\end{equation}
It follows from straightforward calculation that
\begin{equation*}
\label{EQVAR}  
\tau_\alpha^2(\ta) =\sigma_\alpha^2(\ta) - \Bigl(\frac{\alpha\ta}{1-\alpha}\Bigr)(2\vta-\ta).
\end{equation*} 
Consequently, as soon as $\ta \geq 0$, we always have $\tau_\alpha^2(\ta) \leq \sigma_\alpha^2(\ta)$
since $\vta \geq \ta$. 

\begin{thm}
\label{T-LILQSEQUAL1}
Assume that $(\mathcal{A}_1)$ and $(\mathcal{A}_2)$ hold and that the random variable $X$ has a moment of order $>2$. Moreover,
suppose that $f(\ta)>0$ and that the step sequences $(a_n)$ and $(b_n)$ are given by
\begin{equation*}
a_n=\frac{a_1}{n^a}
\hspace{1.5cm}\text{and}\hspace{1.5cm}
b_n=\frac{b_1}{n}
\end{equation*}
where $a_1>0$,  $b_1> 1/2$ and $1/2<a<1$. Then, $(\vtnh)$ and $(\vtnt)$
share the same QSL
\begin{equation}
\label{QSL1}
\lim_{n \rightarrow \infty} \frac{1}{\log n} \sum_{k=1}^n \bigl(  \vtkh - \vta \bigr)^2= \Bigl(\frac{b_1^2}{2b_1-1} \Bigr)\tau^2_{\alpha}(\theta_\alpha)
\hspace{1cm}\text{a.s.}
\end{equation}
%\begin{equation}
%\label{QSL2}
%\lim_{n \rightarrow \infty} \frac{1}{\log n} \sum_{k=1}^n \bigl(  \vtkt - \vta \bigr)^2= \Bigl(\frac{b_1^2}{2b_1-1} \Bigr)\tau^2_{\alpha}(\theta_\alpha)
%\hspace{1cm}\text{a.s.}
%\end{equation}
In addition, they also share the same LIL
\begin{eqnarray}  
\limsup_{n \rightarrow \infty} \left(\frac{n}{2 \log \log n} \right)^{1/2}
 \!\!\!\bigl( \vtnh - \vta \bigr)
&=& - \liminf_{n \rightarrow \infty}
\left(\frac{n}{2 \log\log n}\right)^{1/2} 
\!\!\!\bigl( \vtnh - \vta \bigr) \notag \\
&=& \left(  \frac{b_1^2}{2b_1 -1}  \right)^{1/2} \tau_{\alpha}(\theta_\alpha)
\hspace{1cm}\text{a.s.}
\label{LIL1}
\end{eqnarray}
%\begin{eqnarray}  
%\limsup_{n \rightarrow \infty} \left(\frac{n}{2 \log \log n} \right)^{1/2}
% \!\!\!\bigl( \vtnt - \vta \bigr)
%&=& - \liminf_{n \rightarrow \infty}
%\left(\frac{n}{2 \log\log n}\right)^{1/2} 
%\!\!\!\bigl( \vtnt - \vta \bigr) \notag \\
%&=& \left(  \frac{b_1^2}{2b_1 -1}  \right)^{1/2} \tau_{\alpha}(\theta_\alpha)
%\hspace{1cm}\text{a.s.}
%\label{LIL2}
%\end{eqnarray}
In particular,
\begin{equation*}
\label{LILSUPRM1}  
\limsup_{n \rightarrow \infty} \left(\frac{n}{2 \log \log n} \right)
\bigl( \vtnh - \vta \bigr)^2= \Bigl(\frac{b_1^2}{2b_1 -1}\Bigr) \tau^2_{\alpha}(\theta_\alpha)
\hspace{1cm}\text{a.s.}
\end{equation*} 
%\begin{equation}
%\label{LILSUPRM2}  
%\limsup_{n \rightarrow \infty} \left(\frac{n}{2 \log \log n} \right)
%\bigl( \vtnt - \vta \bigr)^2= \Bigl(\frac{b_1^2}{2b_1 -1} \Bigr)\tau^2_{\alpha}(\theta_\alpha)
%\hspace{1cm}\text{a.s.}
%\end{equation} 
\end{thm}

%\begin{rem}
%It follows from straightforward calculation that
%\begin{equation}
%\label{EQVAR}  
%\tau_\alpha^2(\ta) =\sigma_\alpha^2(\ta) - \Bigl(\frac{\alpha\ta}{1-\alpha}\Bigr)(2\vta-\ta).
%\end{equation} 
%Moreover, one can observe that we always have $\vta \geq \ta$. Consequently, as soon as $\ta \geq 0$, we deduce from
%\eqref{EQVAR} that $\tau_\alpha^2(\ta) \leq \sigma_\alpha^2(\ta)$. It means, from an asymptotic point of view, that if
%$\ta \geq 0$, it is preferable to use of the convexified estimator $\vtnt$.
%\end{rem}

\begin{rem}
In the special case where the step sequence $(b_n)$ is given by
\begin{equation*}
b_n=\frac{1}{n+1},
\end{equation*}
it is easy to see that $\vtnh$ and $\vtnt$ both reduce to
$$
\vtnh=\frac{1}{n} \sum_{k=1}^n \Bigl( \frac{X_{k}}{1-\alpha}\Bigr)\rI_{\{X_{k} >\theta_{k-1}\}}
$$
and
$$
\vtnt=\frac{1}{n} \sum_{k=1}^n \theta_{k-1} +\frac{1}{n} \sum_{k=1}^n \Bigl( \frac{X_{k} -\theta_{k-1} }{1-\alpha}\Bigr)\rI_{\{X_{k} >\theta_{k-1}\}}.
$$
In this setting, we immediately obtain from Theorem \ref{T-LILQSEQUAL1} that
\begin{equation*}
\label{BASICQSL}
\lim_{n \rightarrow \infty} \frac{1}{\log n} \sum_{k=1}^n \bigl(  \vtkh - \vta \bigr)^2= \tau^2_{\alpha}(\theta_\alpha)
\hspace{1cm}\text{a.s.}
\end{equation*}
and
\begin{equation*}
\label{BASICLILSUP}  
\limsup_{n \rightarrow \infty} \left(\frac{n}{2 \log \log n} \right)
\bigl( \vtnh - \vta \bigr)^2=  \tau^2_{\alpha}(\theta_\alpha)
\hspace{1cm}\text{a.s.}
\end{equation*} 
\end{rem}

\begin{thm}
\label{T-LILQSLESS1}
Assume that $(\mathcal{A}_1)$ and $(\mathcal{A}_2)$ hold and that the random variable $X$ has a moment of order $>2$. Moreover,
suppose that $f(\ta)>0$ and that the step sequences $(a_n)$ and $(b_n)$ are given by
\begin{equation*}
a_n=\frac{a_1}{n^a}
\hspace{1.5cm}\text{and}\hspace{1.5cm}
b_n=\frac{b_1}{n^b}
\end{equation*}
where $a_1>0$, $b_1>0$ and $1/2<a<b<1$. Then, $(\vtnh)$ and $(\vtnt)$
share the same QSL
\begin{equation}
\label{QSL3}
\lim_{n \rightarrow \infty} \frac{1}{n^{1-b}} \sum_{k=1}^n \bigl(  \vtkh - \vta \bigr)^2= \Bigl(\frac{b_1}{2(1-b)} \Bigr)\tau^2_{\alpha}(\theta_\alpha)
\hspace{1cm}\text{a.s.}
\end{equation}
%\begin{equation}
%\label{QSL4}
%\lim_{n \rightarrow \infty} \frac{1}{n^{1-b}} \sum_{k=1}^n \bigl(  \vtkt - \vta \bigr)^2= \frac{b_1}{2(1-b)} \tau^2_{\alpha}(\theta_\alpha)
%\hspace{1cm}\text{a.s.}
%\end{equation}
In addition, they also share the same LIL
\begin{eqnarray}  
\limsup_{n \rightarrow \infty} \left(\frac{n^b}{2 (1-b)\log n} \right)^{1/2}
 \!\!\!\bigl( \vtnh - \vta \bigr)
&=& - \liminf_{n \rightarrow \infty}
\left(\frac{n^b}{2 (1-b)\log n}\right)^{1/2} 
\!\!\!\bigl( \vtnh - \vta \bigr) \notag \\
&=& \left(  \frac{b_1}{2}  \right)^{1/2} \tau_{\alpha}(\theta_\alpha)
\hspace{1cm}\text{a.s.}
\label{LIL3}
\end{eqnarray}
%\begin{eqnarray}  
%\limsup_{n \rightarrow \infty} \left(\frac{n^b}{2  \log n} \right)^{1/2}
% \!\!\!\bigl( \vtnt - \vta \bigr)
%&=& - \liminf_{n \rightarrow \infty}
%\left(\frac{n^b}{2 \log n}\right)^{1/2} 
%\!\!\!\bigl( \vtnt - \vta \bigr) \notag \\
%&=& \left(  \frac{b_1(1-b)}{2}  \right)^{1/2} \tau_{\alpha}(\theta_\alpha)
%\hspace{1cm}\text{a.s.}
%\label{LIL4}
%\end{eqnarray}
In particular,
\begin{equation*}
\label{LILSUPRM3}  
\limsup_{n \rightarrow \infty} \left(\frac{n^b}{2 (1-b)\log n} \right)
\bigl( \vtnh - \vta \bigr)^2= \Bigl(\frac{b_1}{2} \Bigr)\tau^2_{\alpha}(\theta_\alpha)
\hspace{1cm}\text{a.s.}
\end{equation*} 
%\begin{equation}
%\label{LILSUPRM4}  
%\limsup_{n \rightarrow \infty} \left(\frac{n^b}{2  \log n} \right)
%\bigl( \vtnt - \vta \bigr)^2= \frac{b_1(1-b)}{2} \tau^2_{\alpha}(\theta_\alpha)
%\hspace{1cm}\text{a.s.}
%\end{equation} 
\end{thm}

\begin{rem} Similar computations in the case where $1/2<b<a<1$ would lead to the same results for the convexified algorithm $(\vtnt)$. However, for the 
standard algorithm $(\vtnh)$, it is necessary to replace the asymptotic variance $\tau_\alpha^2(\ta)$ by $\sigma_\alpha^2(\ta)$. This emphasizes 
the interest of using the convexified algorithm.
\end{rem}

We  now focus our attention on the asymptotic normality of our two-time-scale stochastic algorithms \eqref{TTSALGO1} and \eqref{TTSALGO2}.

\begin{thm}
\label{T-AN}
Assume that $(\mathcal{A}_1)$ and $(\mathcal{A}_2)$ hold and that the random variable $X$ has a moment of order $>2/a$. Moreover,
suppose that $f(\ta)>0$ and that the step sequences $(a_n)$ and $(b_n)$ are given by
\begin{equation*}
a_n=\frac{a_1}{n^a}
\hspace{1.5cm}\text{and}\hspace{1.5cm}
b_n=\frac{b_1}{n^b}
\end{equation*}
where $a_1>0$, $b_1 >0$ and $1/2<a<b \leq 1$ with $b_1>1/2$ if $b=1$. Then, $(\vtnh)$ and $(\vtnt)$
share the same joint asymptotic normality
\begin{equation}
\label{AN1}
\begin{pmatrix}
\sqrt{n^a} \bigl(\theta_n - \ta\bigr) \vspace{1ex} \\
\sqrt{n^b} \bigl( \vtnh - \vta\bigr) \\
\end{pmatrix} \liml \cN\left(0,  \begin{pmatrix}
\Gamma_{\ta} & 0 \\
0 & \Gamma_{\vta} \\
\end{pmatrix}\right)
\end{equation}
where the asymptotic variances are given by
\begin{equation*}
\Gamma_{\ta}= \frac{a_1\alpha(1-\alpha)}{2f(\ta)}
\end{equation*}
and
\begin{equation*}
\Gamma_{\vta}= \left \{
\begin{array}[c]{ccc}
{\displaystyle \frac{b_1^2 \tau^2_{\alpha}(\theta_\alpha)}{2b_1 - 1}}  & \text{if} & b=1, \vspace{1ex} \\
{\displaystyle \frac{b_1 \tau^2_{\alpha}(\theta_\alpha)}{2}} & \text{if} & b<1.
\end{array}
\right.
\end{equation*}
\end{thm}

\begin{rem}
One can observe that the asymptotic covariance matrix in \eqref{AN1} is diagonal. It means that, at the limit, the two algorithms for quantile and superquantile 
estimation are no longer correlated. This is due to the fact that we use two different time scales contrary to Bardou et al. \cite{Bardou2009}. 
Moreover, in the special case where $b=1$, we also recover the same asymptotic variance as the one obtained in \cite{Bardou2009} for the averaged version
of their one-time-scale stochastic algorithm.
\end{rem}

\begin{rem}
The asymptotic variance $\tau^2_{\alpha}(\theta_\alpha)$ can be estimated by 
$$
\tau_n^2= \frac{1}{n}\sum_{k=1}^n \Bigl(\frac{X_k - \theta_{k-1}}{1 - \alpha} \Bigr)^2 \rI_{\{X_{k} >\theta_{k-1}\}}
- \Bigl( \frac{1}{n}\sum_{k=1}^n \Bigl(\frac{X_k - \theta_{k-1}}{1 - \alpha} \Bigr) \rI_{\{X_{k} >\theta_{k-1}\}} \Bigr)^2.
$$ 
Via the same lines as in the proof of the almost sure convergences \eqref{ASCVGSQ1} and \eqref{ASCVGSQ2}, one can verify that
$\tau_n^2 \rightarrow \tau^2_{\alpha}(\theta_\alpha)$ a.s. Therefore, using Slutsky's Theorem, we deduce from \eqref{AN1} that
$(\vtnh)$ and $(\vtnt)$ share the same asymptotic normality
\begin{equation}
\label{AN2}
\sqrt{n^b} \Bigl( \frac{\vtnh - \vta}{\tau_n} \Bigr) \liml \cN(0,\nu^2)
\end{equation}
where
\begin{equation*}
\nu^2= \left \{
\begin{array}[c]{ccc}
{\displaystyle \frac{b_1^2}{2b_1 - 1}}  & \text{if} & b=1, \vspace{1ex} \\
{\displaystyle \frac{b_1}{2}} & \text{if} & b<1.
\end{array}
\right.
\end{equation*}
Convergence \eqref{AN2} allows us to construct asymptotic confidence intervals for the superquantile $\vta$. 
\end{rem}

%%%%%%%%%%%%%%%%%%%%%%%%%%%%%%%%%%%%%%%%%%%%%%%%%%%%%%%%%%%%%%%%%%%%%%%%%%%%%%%%%%%%%%%%%%%%%%%%%%

\section{Our martingale approach}
\label{S-MA}

%%%%%%%%%%%%%%%%%%%%%%%%%%%%%%%%%%%%%%%%%%%%%%%%%%%%%%%%%%%%%%%%%%%%%%%%%%%%%%%%%%%%%%%%%%%%%%%%%%

All our analysis relies on a decomposition of our estimates as sum of a martingale increment and a drift term.
More precisely, it follows from \eqref{TTSALGO1} and \eqref{TTSALGO2} that for all $n \geq 1$,
\begin{equation}
\label{TTSALGOYZ}
\left \{
\begin{array}[c]{ccc}
\vtnph & = & (1-b_n)\vtnh+b_n Y_{n+1} \vspace{2ex} \\
\vtnpt & = & (1-b_n)\vtnt+b_n Z_{n+1}
\end{array}
\right.
\end{equation}
where 
$$
Y_{n+1}=\frac{X_{n+1}}{1-\alpha}\rI_{\{X_{n+1} >\tn\}} 
$$
and
$$
Z_{n+1}=\tn+ \frac{(X_{n+1}-\tn)}{1-\alpha}\rI_{\{X_{n+1} >\tn\}}.
$$
Let $H_\alpha(\theta)$ and $L_\alpha(\theta)$ be the functions defined, for all $\theta \in \dR$, by
\begin{equation}
\label{DEFHL}
H_\alpha(\theta)=\frac{1}{1-\alpha} \dE[X \rI_{\{X >\theta\}}] \hspace{0.5cm}\text{and}\hspace{0.5cm}
L_\alpha(\theta)= \theta +\frac{1}{1-\alpha} \dE[(X-\theta) \rI_{\{X >\theta\}}].
\end{equation}
We clearly have that almost surely
$$\dE[Y_{n+1} | \cF_n]=H_\alpha(\tn) \hspace{1cm}\text{and}\hspace{1cm} \dE[Z_{n+1} | \cF_n]=L_\alpha(\tn).$$
It allows use to split $Y_{n+1}$ and $Z_{n+1}$ as sum of a martingale increment and a drift term,
$Y_{n+1}=\veps_{n+1} + H_\alpha(\tn)$ and  $Z_{n+1}=\xi_{n+1} + L_\alpha(\tn)$. One can also verify that
$\dE[\veps_{n+1}^2 | \cF_n]=\sigma_\alpha^2(\tn)$ and
$\dE[\xi_{n+1}^2 | \cF_n]=\tau_\alpha^2(\tn)$ where the two variances are given by 
\eqref{DEFVAR}.
Then, we immediately deduce from \eqref{TTSALGOYZ} that for all $n \geq 1$,
\begin{equation}
\label{TTSALGOE}
\left \{
\begin{array}[c]{ccc}
\vtnph & = & (1-b_n)\vtnh+b_n (\veps_{n+1} + H_\alpha(\tn)) \vspace{2ex} \\
\vtnpt & = & (1-b_n)\vtnt+b_n (\xi_{n+1}+L_\alpha(\tn)).
\end{array}
\right.
\end{equation}
Hereafter, assume for the sake of simplicity that for all $n \geq 1$, $b_n<1$, since this is true for $n$ large enough.
Let $(P_n)$ be the increasing sequence of positive real numbers defined by
\begin{equation}
\label{DEFPN}
P_n= \prod_{k=1}^n (1-b_k)^{-1}
\end{equation}
with the convention that $P_0\!=\!1$.
Since $(1-b_n)P_n=P_{n-1}$, we obtain from \eqref{TTSALGOE} that 
\begin{equation*}
\left \{
\begin{array}[c]{ccc}
P_n\vtnph & = & P_{n-1}\vtnh+P_nb_n (\veps_{n+1} + H_\alpha(\tn)) \vspace{2ex} \\
P_n\vtnpt & = & P_{n-1}\vtnt+P_n b_n (\xi_{n+1}+L_\alpha(\tn))
\end{array}
\right.
\end{equation*}
which implies the martingale decomposition
\begin{equation}
\label{DECMART}
\left \{
\begin{array}[c]{ccc}
 \vtnph & = & {\displaystyle \frac{1}{P_{n}}\Bigl( \wh{\vartheta}_1+M_{n+1} + H_{n+1} \Bigr)} \vspace{1ex}\\
\vtnpt & = &{\displaystyle \frac{1}{P_{n}}\Bigl( \widetilde{\vartheta}_1+N_{n+1} + L_{n+1} \Bigr)}
\end{array}
\right.
\end{equation}
where
\begin{equation}
\label{DEFMART}
M_{n+1}=\sum_{k=1}^{n} b_k P_k \veps_{k+1}, \hspace{1.5cm}N_{n+1}=\sum_{k=1}^{n} b_k P_k \xi_{k+1}
\end{equation}
\begin{equation}
\label{DEFADDT}
H_{n+1}= \sum_{k=1}^{n} b_k P_k H_\alpha(\theta_k), \hspace{1.5cm}L_{n+1}= \sum_{k=1}^{n} b_k P_k L_\alpha(\theta_k).
\end{equation}
Our strategy is to establish the asymptotic behavior of the two martingales $(M_n)$ and $(N_n)$ 
as well as to determine the crucial role played by the two drift terms $(H_n)$ and $(L_n)$. Several results in our analysis rely on the following keystone lemma 
which concerns the convexity properties of the functions $H_\alpha$ and $L_\alpha$ defined in \eqref{DEFHL}.

\begin{lem}
\label{L-HL}
Assume that $(\mathcal{A}_1)$ and $(\mathcal{A}_2)$ hold. Then, $L_\alpha$ is a convex function such that 
$L_\alpha(\ta)=\vta$, $L_\alpha^\prime(\ta)=0$ and that for all $ \theta \in \dR$,
\begin{equation}
\label{TAYLORL}
0 \leq L_\alpha(\theta)  - L_\alpha(\ta)\leq \frac{ ||f||_\infty}{2(1- \alpha)} \bigl( \theta - \ta  \bigr)^2.
\end{equation}
In addition, we also have $H_\alpha(\ta)=\vta$ and that for all $ \theta \in \dR$,
\begin{equation}
\label{TAYLORH}
\Bigl| H_\alpha(\theta)  - H_\alpha(\ta) + \frac{\ta f(\ta)}{1-\alpha} (\theta -\ta )\Bigr|\leq \frac{ ||\Phi||_\infty}{2(1- \alpha)} \bigl( \theta - \ta  \bigr)^2.
\end{equation}
\end{lem}

\begin{proof}
It follows from \eqref{DEFHL} that for all $\theta \in \dR$,
$$
L_\alpha^\prime(\theta)= \frac{F(\theta)-\alpha}{1-\alpha} 
\hspace{1cm}\text{and}\hspace{1cm}
L_\alpha^{\prime \prime}(\theta)= \frac{f(\theta)}{1-\alpha} .
$$
Consequently, $L_\alpha$ is a convex function such that $L_\alpha^\prime(\ta)=0$. Hence, we deduce from a 
Taylor expansion with integral remainder that for all $ \theta \in \dR$,
\begin{equation*}
L_\alpha(\theta)  = L_\alpha(\ta) + (\theta - \ta)^2 \int_0^1 (1-t) L_\alpha^{\prime \prime}(\theta_t)dt
\end{equation*}
where $\theta_t=\ta+t(\theta - \ta)$, which immediately leads to \eqref{TAYLORL} using $(\mathcal{A}_1)$. Unfortunately,
$H_\alpha$ is not a convex function. However, we obtain from \eqref{DEFHL} that for all $\theta \in \dR$,
$$
H_\alpha^\prime(\theta)= -\frac{\theta f(\theta)}{1-\alpha} 
\hspace{1cm}\text{and}\hspace{1cm}
H_\alpha^{\prime \prime}(\theta)= -\frac{\Phi(\theta)}{1-\alpha} 
$$
where $\Phi(\theta)=f(\theta)+ \theta f^\prime(\theta)$. Finally, \eqref{TAYLORH} follows once again from a Taylor expansion with integral remainder
together with $(\mathcal{A}_2)$.
\end{proof}

\noindent
We have just seen that the function $H_\alpha$ is not convex. Consequently, in order to prove sharp asymptotic properties 
for the sequence $(\vtnh)$, it is necessary to slightly modify the first martingale decomposition in \eqref{DECMART}. For all $\theta \in \dR$, let
\begin{equation}
\label{DEFREMAINDERS}
\left \{
\begin{array}[c]{ccc}
G_\alpha(\theta)  & = & F(\theta) - \alpha -f(\ta)(\theta - \ta),\vspace{2ex}\\
R_\alpha(\theta)  & = & H_\alpha(\theta) - \vta +C_\alpha (\theta - \ta)
\end{array}
\right.
\end{equation}
where
$$
C_\alpha=- H_\alpha^\prime(\ta)= \frac{\ta f(\ta)}{1-\alpha}.
$$
We deduce from \eqref{TTSALGO1}, \eqref{TTSALGOE} and \eqref{DEFREMAINDERS} that for all $n \geq 1$,
\begin{equation}
\label{DECALGON1}
\left \{
\begin{array}[c]{ccl}
 \tnp -\ta \!& \!\!=\! \!&\! (1-a_n f(\ta)) (\tn - \ta)-a_n \bigl( V_{n+1} + G_\alpha(\theta_n) \bigr), \vspace{2ex}\\
\vtnph -\vta \!& \!\!=\! \!&\!  (1-b_n)(\vtnh -\vta)+b_n \bigl(\veps_{n+1} \!+\! R_\alpha(\tn) \!-\! C_\alpha (\theta_n - \ta)\bigr)
\end{array}
\right.
\end{equation}
with $V_{n+1}=\rI_{\{X_{n+1} \leq \tn \}} - F(\theta_n)$. Hereafter, we shall consider a tailor-made weighted sum of our estimates given by
$\Delta_1=0$ and, for all $n\geq 2$,
\begin{equation}
\label{DIFDELTA}
\Delta_n = (\vtnh -\vta ) - \delta_n (\tn -\ta)
\end{equation}
where $(\delta_n)$ is a deterministic sequence, depending on $(a_n)$ and $(b_n)$, which will be explicitly given below.
It follows from \eqref{DECALGON1} together with straightforward calculation that for all $n\geq 2$, 
\begin{equation}
\label{DECALGON2}
\Delta_{n+1}=(1-b_n) \Delta_n +b_n \Bigl(W_{n+1}+ R_\alpha(\tn) +\frac{a_n}{b_n}\delta_{n+1} G_\alpha(\theta_n) +\nu_{n+1} (\tn - \ta) \Bigr)
\end{equation}
where
\begin{equation}
\label{DEFWN}
W_{n+1}= \veps_{n+1}+\frac{a_n}{b_n}\delta_{n+1}V_{n+1}
\end{equation}
and
\begin{equation}
\label{DEFNUN}
\nu_{n+1} = \frac{1}{b_n}\Bigl((1-b_n)\delta_n -  C_\alpha b_n-\delta_{n+1}(1-a_n f(\ta))\Bigr).
\end{equation}
We have several strategies in order to simplify the expression of $\nu_{n+1}$. A first possibility
that cancels several terms in \eqref{DEFNUN} is to choose
$$
\delta_{n+1}=\frac{C_\alpha  b_n}{f(\ta) a_n}.
$$
It clearly reduces $\nu_{n+1}$ to
$$
\nu_{n+1}=\frac{(1-b_n)\delta_n - \delta_{n+1}}{b_n}.
$$
Another more sophisticated choice, which only works if $a_n f(\theta_\alpha) -b_n\neq 0$, is to take
\begin{equation}
\label{DEFDELTAN}
\delta_{n+1}= \frac{C_\alpha  b_n}{a_n f(\theta_\alpha) -b_n}.
\end{equation}
It implies that
\begin{equation}
\label{DEFNUNN}
\nu_{n+1}=\frac{(1-b_n)(\delta_n - \delta_{n+1})}{b_n}.
\end{equation}
The two choices are quite similar in the special case where $b_n=b_1/n$. However, in the case where the step
$b_n=b_1/n^b$ with $a<b<1$, the second choice outperforms the first one as $\nu_n$ goes faster towards zero
as $n$ grows to infinity. Throughout the sequel, we shall make use of the second choice given by \eqref{DEFDELTAN}.
We deduce from \eqref{DECALGON2} the new martingale decomposition
\begin{equation}
\label{DECALGON3}
\Delta_{n+1} =  \frac{1}{P_{n}}\Bigl( \cM_{n+1} +\cH_{n+1} +\cR_{n+1}\Bigr)
\end{equation}
where
\begin{equation}
\label{DEFMARTNEW}
\cM_{n+1}=\sum_{k=1}^{n} b_k P_k W_{k+1}, \hspace{1.5cm}\cH_{n+1}=\sum_{k=1}^{n} b_k P_k \nu_{k+1}\bigl(\theta_k - \ta\bigr) 
\end{equation}
and
\begin{equation}
\label{DEFRNEW}
\cR_{n+1}= \sum_{k=1}^{n} b_k P_k \Bigl(  R_\alpha(\theta_k) + \frac{a_k}{b_k}\delta_{k+1}G_\alpha(\theta_k)\Bigr). 
\end{equation}

%%%%%%%%%%%%%%%%%%%%%%%%%%%%%%%%%%%%%%%%%%%%%%%%%%%%%%%%%%%%%%%%%%%%%%%%%%%%%%%%%%%%%%%%%%%%%%%%%%

\section{Proofs of the almost sure convergence results}
\label{S-PRASCVG}

%%%%%%%%%%%%%%%%%%%%%%%%%%%%%%%%%%%%%%%%%%%%%%%%%%%%%%%%%%%%%%%%%%%%%%%%%%%%%%%%%%%%%%%%%%%%%%%%%%

\subsection{The basic almost sure properties.}
The starting point in our analysis of the almost sure convergence of our estimates is the following lemma.

\begin{lem}
\label{L-ASCVG}
Assume that $(\mathcal{A}_1)$ holds and that the random variable $X$ is square integrable.
Then, we have the almost sure convergences
\begin{equation}
\label{ASCVGMN}
\lim_{n \rightarrow \infty} \frac{M_{n+1}}{P_{n}} = 0 
\hspace{1cm}\text{and}\hspace{1cm}
\lim_{n \rightarrow \infty} \frac{N_{n+1}}{P_{n}} = 0 
\hspace{1cm}\text{a.s.}
\end{equation}
\end{lem}

\begin{proof}
Let $(\Sigma_n^{\varepsilon})$ and $(\Sigma_n^{\xi})$ be the two locally square integrable martingales 
$$
\Sigma_n^{\varepsilon}= \sum_{k=1}^{n-1} b_k\veps_{k+1},
\hspace{1.5cm} 
\Sigma_n^{\xi}= \sum_{k=1}^{n-1} b_k\xi_{k+1}.
$$
Their predictable quadratic variations \cite{Duflo1997} are
respectively given by
$$
\langle \Sigma^{\varepsilon} \rangle_n= \sum_{k=1}^{n-1} b_k^2 \sigma_\alpha^2(\theta_k) 
\hspace{1cm} \text{and} \hspace{1cm}
\langle \Sigma^{\xi} \rangle_n=\sum_{k=1}^{n-1} b_k^2 \tau_\alpha^2(\theta_k). 
$$
It follows from convergence \eqref{ASCVGRM} and the continuity of the variances $\sigma_\alpha^2(\theta)$ and
$\tau_\alpha^2(\theta)$ given by \eqref{DEFVAR} that $\sigma_\alpha^2(\tn) \longrightarrow \sigma_\alpha^2(\ta)$
and $\tau_\alpha^2(\tn) \longrightarrow \tau_\alpha^2(\ta)$ a.s.
%\begin{equation}
%\label{ASCVGSIG}
%\lim_{n \rightarrow \infty} \sigma_\alpha^2(\tn) = \sigma_\alpha^2(\ta) 
%\hspace{1cm}\text{a.s}
%\end{equation}
%and
%\begin{equation}
%\label{ASCVGTAUX}
%\lim_{n \rightarrow \infty} \tau_\alpha^2(\tn) = \tau_\alpha^2(\ta) 
%\hspace{1cm}\text{a.s}
%\end{equation}
Consequently, we get from 
%\eqref{ASCVGSIG}, \eqref{ASCVGTAUX} and 
the right-hand side of \eqref{CONDSTEP} that 
\begin{equation}
\label{ASCVGIPM}
\lim_{n \rightarrow \infty} \langle \Sigma^{\varepsilon} \rangle_n < + \infty 
\hspace{1cm} \text{and} \hspace{1cm}
\lim_{n \rightarrow \infty} \langle \Sigma^{\xi} \rangle_n < + \infty 
\hspace{1cm}\text{a.s}
\end{equation}
Therefore, we obtain from the strong law of large numbers for martingales given e.g. by theorem 1.3.24 in \cite{Duflo1997} that
$(\Sigma_n^{\varepsilon})$ and $(\Sigma_n^{\xi})$ both converge almost surely.
The rest of the proof proceeds in a standard way with the help of Kronecker's lemma. 
As a matter of fact, we can deduce from the left-hand side of \eqref{CONDSTEP} that
the sequence $(P_n)$, defined in \eqref{DEFPN}, is strictly increasing to infinity. 
In addition, we just showed the almost sure convergence of the series
$$
\sum_{n=1}^{\infty} b_n\veps_{n+1}
\hspace{1cm} \text{and} \hspace{1cm}
\sum_{n=1}^{\infty} b_n\xi_{n+1}.
$$
Consequently, we immediately deduce from Kronecker's lemma that
\begin{equation*}
\lim_{n \rightarrow \infty} \frac{1}{P_{n}} \sum_{k=1}^{n} b_k P_k \veps_{k+1}= 0 
\hspace{1cm}\text{and}\hspace{1cm}
\lim_{n \rightarrow \infty} \frac{1}{P_{n}} \sum_{k=1}^{n} b_k P_k \xi_{k+1}= 0 
\hspace{1cm}\text{a.s.}
\end{equation*}
which is exactly what we wanted to prove.
\end{proof}

\noindent{\bf Proof of Theorem \ref{T-ASCVGSQ}.}
%%%%%%%%%%%%%%%%%%%%%%%%%%%%%%%%%%%%%%%%%%%%%%%%%%%%%%%%%%%%%%%%%%%%%%%%%%%%%%%%%%%%%%%%%%%%%%%%%%
We recall from \eqref{DECMART} that for all $n \geq 1$,
\begin{equation*}
\left \{
\begin{array}[c]{ccc}
 \vtnph & = & {\displaystyle \frac{1}{P_{n}}\Bigl( \wh{\vartheta}_1+M_{n+1} + H_{n+1} \Bigr)} \vspace{1ex}\\
\vtnpt & = &{\displaystyle \frac{1}{P_{n}}\Bigl( \widetilde{\vartheta}_1+N_{n+1} + L_{n+1} \Bigr)}.
\end{array}
\right.
\end{equation*}
We have from \eqref{ASCVGRM} together with the continuity of the functions $H_\alpha$ and $L_\alpha$ that
\begin{equation}
\label{ASCVGH}
\lim_{n \rightarrow \infty} H_\alpha(\tn) = H_\alpha(\ta) 
\hspace{1cm}\text{a.s}
\end{equation}
and
\begin{equation}
\label{ASCVGL}
\lim_{n \rightarrow \infty} L_\alpha(\tn) = L_\alpha(\ta) 
\hspace{1cm}\text{a.s}
\end{equation}
One can observe that $H_\alpha(\ta)= L_\alpha(\ta) =\vta$. 
Moreover, it is easy to see that for all $n\geq 1$, $b_nP_n=P_n-P_{n-1}$.
Hence, we obtain by a telescoping argument that
\begin{equation}
\label{DECPN}
\sum_{k=1}^{n} b_k P_k=P_n -P_0,
\end{equation}
which leads to
$$
\lim_{n \rightarrow \infty} \frac{1}{P_n}\sum_{k=1}^{n} b_k P_k=1.
$$
Therefore, it follows from Toeplitz's lemma that 
\begin{equation}
\label{ASCVGHL}
\lim_{n \rightarrow \infty} \frac{H_{n+1}}{P_{n}} = \vta 
\hspace{1cm}\text{and}\hspace{1cm}
\lim_{n \rightarrow \infty} \frac{L_{n+1}}{P_{n}} = \vta 
\hspace{1cm}\text{a.s.}
\end{equation}
Finally, we find from \eqref{DECMART}, \eqref{ASCVGMN} and \eqref{ASCVGHL} that
\begin{equation}
\label{ASCVGVAR}
\lim_{n \rightarrow \infty}  \vtnh = \vta 
\hspace{1cm}\text{and}\hspace{1cm}
\lim_{n \rightarrow \infty}  \vtnt = \vta 
\hspace{1cm}\text{a.s.}
\end{equation}
which completes the proof of Theorem \ref{T-ASCVGSQ}.
\demend

\subsection{A keystone lemma.}
The QSL as well as the LIL for our estimates require the sharp asymptotic behavior of the
sequence $(P_n)$ defined in \eqref{DEFPN}. Surprisingly, to the best of our knowledge, the following keystone lemma is new. It involves
the famous Euler-Riemann zeta function.

%%%%%%%%%%%%%%%%%%%%%%%%%%%%%%%%%%%%%%%%%%%%%%%%%%%%%%%%%%%%%%%%%%%%%%%%%%%%%%%%%%%%%%%%%%%%%%%%%%

\begin{lem}
\label{L-CVGPN}
Assume that for some $0<b_1<1$,
\begin{equation}
\label{DEFPN1}
P_n=\prod_{k=1}^n \Bigl( 1 - \frac{b_1}{k} \Bigr)^{-1}.
\end{equation}
Then, we have
\begin{equation}
\label{CVGPNGAMMA}
\lim_{n \rightarrow \infty} \frac{1}{n^{b_1}}P_n= \Gamma(1-b_1)
\end{equation}
where $\Gamma$ stands for the Euler gamma function.
Moreover, suppose that
\begin{equation}
\label{DEFPNLESS1}
P_n=\prod_{k=1}^n \Bigl( 1 - \frac{b_1}{k^b} \Bigr)^{-1}
\end{equation}
where $1/2<b<1$. Then, we have 
\begin{equation}
\label{CVGPNZETA}
\lim_{n \rightarrow \infty} \frac{1}{\exp(cn^{1-b})}P_n= \exp(\Lambda)
\end{equation}
with $c=b_1/(1-b)$ and the limiting value
$$
\Lambda=\sum_{n=2}^\infty \frac{b_1^n}{n} \zeta(b n)
$$
where $\zeta$ stands for the Riemann zeta function.
\end{lem}

\begin{rem}
The link between the first case $b=1$ and the second case $1/2<b<1$ is given the following formula due to Euler.
For all $|x|<1$,
$$
\log \Gamma(1-x)=\gamma x + \sum_{n=2}^\infty \frac{x^n}{n} \zeta(n)
$$
where $\gamma$ is the Euler-Mascheroni constant.
\end{rem}

\begin{rem}
The case $b_1 \geq 1$ can be treated in the same way. For example, concerning the first part of
Lemma \ref{L-CVGPN}, it is only necessary to replace $P_n$ defined in \eqref{DEFPN1} by
$$
P_n=\prod_{k=1+ \lfloor b_1 \rfloor}^n \Bigl( 1 - \frac{b_1}{k} \Bigr)^{-1}
$$
where $\lfloor b_1 \rfloor$ is the integer part of $b_1$. Then, we obtain that
\begin{equation*}
\lim_{n \rightarrow \infty} \frac{1}{n^{b_1}}P_n= \frac{\Gamma(1-\{b_1\})}{\Gamma(1+ \lfloor b_1 \rfloor)}
\end{equation*}
where $\{b_1\}=b_1-\lfloor b_1 \rfloor$ stands for the fractional part of $b_1$.
\end{rem}

\begin{proof}
In the first case $b=1$, we clearly have
\begin{equation}
\label{PGAMMA}
P_n=\prod_{k=1}^n \Bigl( 1 - \frac{b_1}{k} \Bigr)^{-1}= \frac{\Gamma(n+1) \Gamma(1-b_1)}{\Gamma(n+1-b_1)}.
\end{equation}
It is well-known that for any $c>0$,
\begin{equation}
\label{CVGAMMA}
\lim_{n \rightarrow \infty} \frac{\Gamma(n+c)}{\Gamma(n) n^c}= 1.
\end{equation}
Hence, we obtain from \eqref{PGAMMA} and \eqref{CVGAMMA} that
\begin{equation}
\label{CVGPN}
\lim_{n \rightarrow \infty} \frac{1}{n^{b_1}}P_n= \Gamma(1-b_1).
\end{equation}
The second case $1/2<b<1$ is much more difficult to handle.
It follows from the Taylor expansion of the natural logarithm
$$
\log(1-x)=-\sum_{\ell=1}^\infty \frac{x^\ell}{\ell}
$$
that
\begin{eqnarray}
\log(P_n) &=&-\sum_{k=1}^n\log\Bigl(1- \frac{b_1}{k^b}\Bigr)=\sum_{k=1}^n \sum_{\ell=1}^\infty \frac{1}{\ell} \Bigl(\frac{b_1}{k^b}\Bigr)^\ell
= \sum_{\ell=1}^\infty \sum_{k=1}^n \frac{1}{\ell} \Bigl(\frac{b_1}{k^b}\Bigr)^\ell,  \notag \\
&=& b_1 \sum_{k=1}^n \frac{1}{k^b} + \sum_{\ell=2}^\infty \frac{b_1^\ell}{\ell}\sum_{k=1}^n \frac{1}{k^{b\ell}}.
\label{LOGPN}
\end{eqnarray}
It is well-known that
$$
\lim_{n \rightarrow \infty}\frac{1}{n^{1-b}}\sum_{k=1}^n \frac{1}{k^b}=\frac{1}{1-b}.
$$
In addition, as $b>1/2$, we always have for all $\ell \geq 2$, $b\ell >1$. Consequently,
$$
\lim_{n \rightarrow \infty} \sum_{k=1}^n \frac{1}{k^{b\ell}} =\zeta(b \ell)
$$
where $\zeta$ is the Riemann zeta function. Therefore, we obtain from \eqref{LOGPN} that
\begin{equation}
\label{CVGPNLESS1}
\lim_{n \rightarrow \infty} \frac{1}{\exp(cn^{1-b})}P_n= \exp(\Lambda)
\end{equation}
where $c=b_1/(1-b)$ and the limiting value
$$
\Lambda=\sum_{\ell=2}^\infty \frac{b_1^\ell}{\ell} \zeta(b \ell).
$$
\end{proof}

\subsection{The fast step size case.}
The proof of Theorem  \ref{T-LILQSEQUAL1} relies on the following lemma which provides the QSL and the LIL
for the martingales $(\cM_n)$ and $(N_n)$.

%%%%%%%%%%%%%%%%%%%%%%%%%%%%%%%%%%%%%%%%%%%%%%%%%%%%%%%%%%%%%%%%%%%%%%%%%%%%%%%%%%%%%%%%%%%%%%%%%%

\begin{lem}
\label{L-MART1}
Assume that the step sequences $(a_n)$ and $(b_n)$ are given by
\begin{equation*}
a_n=\frac{a_1}{n^a}
\hspace{1.5cm}\text{and}\hspace{1.5cm}
b_n=\frac{b_1}{n}
\end{equation*}
where $a_1>0$,  $b_1>1/2$ and $1/2<a<1$. Then, $(\cM_n)$ and $(N_n)$ share the same QSL
\begin{equation}
\label{QSLMARTM1}
\lim_{n \rightarrow \infty} \frac{1}{\log n} \sum_{k=1}^n \Bigl(  \frac{\cM_k}{P_{k-1}} \Bigr)^2= \Bigl(\frac{b_1^2}{2b_1-1} \Bigr)\tau^2_{\alpha}(\theta_\alpha)
\hspace{1cm}\text{a.s.}
\end{equation}
%\begin{equation}
%\label{QSLMARTN1}
%\lim_{n \rightarrow \infty} \frac{1}{\log n} \sum_{k=1}^n \Bigl(  \frac{N_k}{P_{k-1}} \Bigr)^2= \Bigl(\frac{b_1^2}{2b_1-1} \Bigr)\tau^2_{\alpha}(\theta_\alpha)
%\hspace{1cm}\text{a.s.}
%\end{equation}
In addition, they also share the same LIL
\begin{eqnarray}  
\limsup_{n \rightarrow \infty} \left(\frac{n}{2  \log \log n} \right)^{1/2}
 \!\!\!\Bigl(\frac{\cM_n}{P_{n-1}}\Bigr)
&=& -  \liminf_{n \rightarrow \infty} \left(\frac{n}{2  \log \log n} \right)^{1/2}
 \!\!\!\Bigl(\frac{\cM_n}{P_{n-1}}\Bigr)  \notag \\
&=& \left(  \frac{b_1^2}{2b_1 -1}  \right)^{1/2} \tau_{\alpha}(\theta_\alpha)
\hspace{1cm}\text{a.s.}
\label{LILMARTN1}
\end{eqnarray}
%\begin{eqnarray}  
%\limsup_{n \rightarrow \infty} \left(\frac{n}{2  \log \log n} \right)^{1/2}
% \!\!\!\Bigl(\frac{N_n}{P_{n-1}}\Bigr)
%&=& - \left(\frac{n}{2  \log \log n} \right)^{1/2}
% \!\!\!\Bigl(\frac{N_n}{P_{n-1}}\Bigr)  \notag \\
%&=& \left(  \frac{b_1^2}{2b_1 -1}  \right)^{1/2} \tau_{\alpha}(\theta_\alpha)
%\hspace{1cm}\text{a.s.}
%\label{LILMARTN1}
%\end{eqnarray}
%Moreover, $(N_n)$ shares exactly the same quadratic strong law and the same law of the iterated logarithm than $(\cM_n)$.
\end{lem}

\begin{proof}
We first focus our attention on the martingale $(\cM_n)$ defined by
$$
\cM_{n+1}=\sum_{k=1}^{n} b_k P_k W_{k+1}
$$
where
$$
W_{n+1}=\veps_{n+1} +\frac{a_n}{b_n}\delta_{n+1}V_{n+1}
$$
with $V_{n+1}=\rI_{\{X_{n+1} \leq \tn \}} - F(\theta_n)$. We clearly have $\dE[V_{n+1} | \cF_n]=0$, $\dE[W_{n+1} | \cF_n]=0$,
and $\dE[V_{n+1}^2 | \cF_n]=F(\theta_n)(1-F(\theta_n))$,
$\dE[W_{n+1}^2 | \cF_n]=\tau_n^2(\tn)$ where for all $\theta \in \dR$,
\begin{equation}
\label{DEFTAUN}
\tau_n^2(\theta)=\sigma_\alpha^2(\theta)+\Bigl(\frac{a_n \delta_{n+1}}{b_n}\Bigr)^2 F(\theta)(1-F(\theta)) -\Bigl(\frac{2 a_n \delta_{n+1} }{b_n}\Bigr)F(\theta)H_\alpha(\theta).
\end{equation}
We obtain from \eqref{DEFDELTAN} that 
\begin{equation}
\label{CVGRATIO}
\lim_{n \rightarrow \infty}\frac{a_n \delta_{n+1}}{b_n}=\frac{\ta}{1-\alpha}.
\end{equation}
Consequently, we infer from \eqref{ASCVGRM}, \eqref{DEFTAUN} and \eqref{CVGRATIO} that
\begin{equation}
\label{ASCVGNU}
\lim_{n \rightarrow \infty} \tau_n^2(\tn) = \sigma_\alpha^2(\ta) - \Bigl(\frac{\alpha\ta}{1-\alpha}\Bigr)(2\vta-\ta)=\tau_\alpha^2(\ta) 
\hspace{1cm}\text{a.s}
\end{equation}
Hereafter, assume for the sake of simplicity that $1/2<b_1<1$ inasmuch as the proof follows exactly the same lines for $b_1 \geq 1$.
On the one hand, the predictable quadratic variation of $(\cM_n)$ is given by
$$
\ccMc_n = \sum_{k=1}^{n-1} b_k^2 P_k^2 \tau^2_k(\theta_k).
$$
On the other hand, as $b_1\!>\!1/2$, we obtain from convergence \eqref{CVGPNGAMMA} in Lemma \ref{L-CVGPN} that
\begin{equation*}
\lim_{n \rightarrow \infty} \frac{1}{n^{2b_1-1}}\sum_{k=1}^n b_k^2 P_k^2 = \Bigl(\frac{b_1^2}{2b_1-1} \Bigr) \Gamma^2(1-b_1).
\end{equation*}
Then, we deduce from \eqref{ASCVGNU} and Toeplitz's lemma that
\begin{equation}
\label{CVGIPCMN}
\lim_{n \rightarrow \infty} \frac{1}{n^{2b_1-1}} \ccMc_n = \Bigl(\frac{b_1^2}{2b_1-1} \Bigr) \Gamma^2(1-b_1) \tau^2_\alpha(\ta)
\hspace{1cm}\text{a.s.}
\end{equation}
Denote by $f_n$ the explosion coefficient associated with the martingale $(\cM_n)$,
$$
f_n = \frac{\ccMc_n - \langle \cM \rangle_{n-1}}{\ccMc_n}.
$$
We obtain from \eqref{CVGIPCMN} that
\begin{equation}
\label{CVGFN}
\lim_{n \rightarrow \infty} nf_n=2b_1-1 \hspace{1cm}\text{a.s.}
\end{equation}
It means that $f_n$ converges to zero almost surely at rate $n$. In addition, we already saw from \eqref{ASCVGNU} that
$$
\lim_{n \rightarrow \infty}  \dE[ W_{n+1}^2 | \cF_n] = \tau^2_\alpha(\ta)
\hspace{1cm}\text{a.s.}
$$
Furthermore, the random variable $X$ has a moment of order $>2$. It implies that for some real number $p>2$,
$$
\sup_{n \geq 0} \dE[ | W_{n+1} |^p | \cF_n] < \infty \hspace{1cm}\text{a.s.}
$$
Consequently, we deduce from the QSL for martingales given in theorem 3 of \cite{bercu2004} that
\begin{equation}
\label{QSLCMN1}
\lim_{n \rightarrow \infty} \frac{1}{\log \ccMc_n}   \sum_{k=1}^n f_k  \frac{\cM_k^2}{\langle \cM \rangle _k} 
 = 1
\hspace{1cm}\text{a.s.}
\end{equation}
Hence, we obtain from the conjunction of \eqref{CVGIPCMN} and \eqref{QSLCMN1} that
\begin{equation}
\label{QSLCMN2}
\lim_{n \rightarrow \infty} \frac{1}{\log n}   \sum_{k=1}^n   \frac{\cM_k^2}{P_{k-1}^2 }
 = \Bigl(\frac{b_1^2}{2b_1-1} \Bigr)\tau^2_{\alpha}(\theta_\alpha)
\hspace{1cm}\text{a.s.}
\end{equation}
We shall now proceed to the proof of the LIL given 
by \eqref{LILMARTN1}. We find from \eqref{CVGFN} that the explosion coefficient $f_n$ satisfies
$$
\sum_{n=1}^\infty f_n^{p/2} < +\infty
\hspace{1cm}\text{a.s.}
$$
Therefore, we deduce from the LIL for martingales \cite{Stout1970}, see also corollary 6.4.25 in \cite{Duflo1997} that
\begin{eqnarray}   
\limsup_{n \rightarrow \infty} \left(\frac{1}{2 \ccMc_n \log \log \ccMc_n} \right)^{1/2}
 \!\!\!\cM_n
&=& - \liminf_{n \rightarrow \infty}
\left(\frac{1}{2 \ccMc_n \log \log \ccMc_n} \right)^{1/2}
 \!\!\!\cM_n  \notag \\
&=& 1
\hspace{1cm}\text{a.s.}
\label{LILCMN1}
\end{eqnarray}
Hence, it follows from the conjunction of \eqref{CVGPN}, \eqref{CVGIPCMN} and \eqref{LILCMN1} that
\begin{eqnarray*}   
\limsup_{n \rightarrow \infty} \left(\frac{n}{2  \log \log n} \right)^{1/2}
 \!\!\!\Bigl(\frac{\cM_n}{P_{n-1}}\Bigr)
&=& - \liminf_{n \rightarrow \infty}\left(\frac{n}{2  \log \log n} \right)^{1/2}
 \!\!\!\Bigl(\frac{\cM_n}{P_{n-1}}\Bigr)  \notag \\
&=& \left(  \frac{b_1^2}{2b_1 -1}  \right)^{1/2} \tau_{\alpha}(\theta_\alpha)
\hspace{1cm}\text{a.s.}
%\label{LILCMN2}
\end{eqnarray*}
which is exactly what we wanted to prove.
Finally, concerning the martingale $(N_n)$ given by
$$
N_{n+1}=\sum_{k=1}^{n} b_k P_k \xi_{k+1},
$$
the only minor change is that $\dE[\xi_{n+1}^2 | \cF_n]= \tau_\alpha^2(\tn)$. However, we already saw that
$\tau_\alpha^2(\tn) \longrightarrow \tau_\alpha^2(\ta)$ a.s. Consequently, $(\cM_n)$ and $(N_n)$ share 
the same QSL and the same LIL, which completes the proof 
of Lemma \ref{L-MART1}.
\end{proof}

\noindent{\bf Proof of Theorem \ref{T-LILQSEQUAL1}.}
%%%%%%%%%%%%%%%%%%%%%%%%%%%%%%%%%%%%%%%%%%%%%%%%%%%%%%%%%%%%%%%%%%%%%%%%%%%%%%%%%%%%%%%%%%%%%%%%%%
We shall only prove Theorem \ref{T-LILQSEQUAL1} in the special case where $b_n=b_1/n$ with $1/2<b_1<1$ inasmuch as the proof
in the case $b_1 \geq 1$ follows essentially the same lines. First of all, we focus our attention on the standard estimator
$\vtnh$.

$\bullet$
Our strategy is first to establish the QSL for the sequence $(\Delta_n)$ given by \eqref{DIFDELTA} and then to come back to
$\vtnh$. We recall from \eqref{DECALGON3} that for all $n \geq 2$,
\begin{equation*}
\Delta_{n+1} =  \frac{1}{P_{n}}\Bigl( \cM_{n+1} +\cH_{n+1} +\cR_{n+1}\Bigr).
\end{equation*}
We claim that the weighted sequence $(\Delta_n)$ satisfies the QSL
\begin{equation}
\label{QSLDELTAN}
\lim_{n \rightarrow \infty} \frac{1}{\log n} \sum_{k=1}^n \Delta^2_k= \Bigl(\frac{b_1^2}{2b_1-1} \Bigr)\tau^2_{\alpha}(\theta_\alpha)
\hspace{1cm}\text{a.s.}
\end{equation}
As a matter of fact, we already saw from \eqref{QSLMARTM1} that 
\begin{equation*}
\lim_{n \rightarrow \infty} \frac{1}{\log n} \sum_{k=1}^n \Bigl(  \frac{\cM_k}{P_{k-1}} \Bigr)^2= \Bigl(\frac{b_1^2}{2b_1-1} \Bigr)\tau^2_{\alpha}(\theta_\alpha)
\hspace{1cm}\text{a.s.}
\end{equation*}
Hence, in order to prove \eqref{QSLDELTAN}, it is necessary to show that
\begin{equation}
\label{PRQSLHRN1}
\lim_{n \rightarrow \infty} \frac{1}{\log n} \sum_{k=1}^n \Bigl(  \frac{\cH_k}{P_{k-1}} \Bigr)^2 =0
\hspace{0.5cm}\text{and}\hspace{0.5cm}
\lim_{n \rightarrow \infty} \frac{1}{\log n} \sum_{k=1}^n \Bigl(  \frac{\cR_k}{P_{k-1}} \Bigr)^2 =0
\hspace{1cm}\text{a.s.}
\end{equation}
On the one hand, it follows from \eqref{LILSUPRMStepa} that for $n$ large enough and for all $k\geq n$,
\begin{equation}
\label{MAJLILStepa}
\bigl( \theta_k - \ta \bigr)^2 \leq 2 D_a \Bigl(\frac{\log k}{k^a}\Bigr)
\hspace{1cm}\text{a.s.}
\end{equation}
where 
$$
D_a= \frac{2a_1(1-a)\alpha(1-\alpha)}{f(\ta)}.
$$
Consequently, we obtain from \eqref{DEFMARTNEW} and \eqref{MAJLILStepa} that
\begin{equation}
\label{PRCVGHN1}
| \cH_{n+1} | = O \left( \sum_{k=1}^n \frac{b_kP_k | \nu_{k+1}| \sqrt{\log k}}{k^{a/2}} \right)
\hspace{1cm}\text{a.s.}
\end{equation}
Furthermore, one can easily check from \eqref{DEFDELTAN} and \eqref{DEFNUNN} that
$$
\lim_{n \rightarrow \infty}  n^{1-a} \nu_{n+1}= \frac{(1-a)C_\alpha}{a_1f(\ta)}.
$$
In addition, we also recall from convergence \eqref{CVGPNGAMMA} in Lemma \ref{L-CVGPN} that
$$
\lim_{n \rightarrow \infty} \frac{1}{n^{b_1}}P_n= \Gamma(1-b_1).
$$
Hence, we deduce from \eqref{PRCVGHN1} that
\begin{equation}
\label{PRCVGHN2}
| \cH_{n+1} | = O \left( \sum_{k=1}^n \frac{\sqrt{\log k}}{k^{2-b_1-a/2}} \right)
\hspace{1cm}\text{a.s.}
\end{equation}
It follows from \eqref{PRCVGHN2} that
\begin{equation}
\label{PRCVGHN3}
\sum_{n=1}^\infty \Bigl(\frac{\cH_n}{P_{n-1}} \Bigr)^2 <+\infty
\hspace{1cm}\text{a.s.}
\end{equation}
As a matter of fact, let $d=2 - b_1 - a/2$. If $d>1$ that is $b_1<1-a/2$, we obtain from \eqref{PRCVGHN2} 
that $| \cH_{n} | = O\bigl( 1 \bigr)$ a.s. Consequently, as $b_1>1/2$, \eqref{PRCVGHN3} holds true. In addition,
if $d=1$ that is $b_1=1-a/2$, we deduce from \eqref{PRCVGHN2} that $| \cH_{n} | = O\bigl( (\log n)^{3/2} \bigr)$ a.s.
which implies that 
\begin{equation*}
\sum_{n=1}^\infty \Bigl(\frac{\cH_n}{P_{n-1}} \Bigr)^2 =  O \left( \sum_{n=1}^\infty \frac{(\log n)^3}{n^{2b_1}} \right)=O(1)
\hspace{1cm}\text{a.s.}
\end{equation*}
Moreover, if $d<1$ that is $b_1>1-a/2$, we get from \eqref{PRCVGHN2} that $| \cH_{n} | = O\bigl( (\log n)^{1/2}n^{1-d} \bigr)$ a.s.
leading to
\begin{equation*}
\sum_{n=1}^\infty \Bigl(\frac{\cH_n}{P_{n-1}} \Bigr)^2 =  O \left( \sum_{n=1}^\infty \frac{(\log n)}{n^{2b_1+d-1}} \right)=
O \left( \sum_{n=1}^\infty \frac{(\log n)}{n^{2-a}} \right)=O(1)
\hspace{1cm}\text{a.s.}
\end{equation*}
On the other hand, \eqref{DEFRNEW} together with \eqref{TAYLORH} and \eqref{DEFREMAINDERS} imply that
\begin{equation}
\label{PRCVGRN1}
| \cR_{n+1} | = O \left( \sum_{k=1}^n b_kP_k\bigl(\theta_k - \ta\bigr)^2 \right)
\hspace{1cm}\text{a.s.}
\end{equation}
It follows from \eqref{MAJLILStepa} and \eqref{PRCVGRN1} that
\begin{equation}
\label{PRCVGRN2}
| \cR_{n+1} | = O \left( \sum_{k=1}^n \frac{\log k}{k^{1+a-b_1}} \right)
\hspace{1cm}\text{a.s.}
\end{equation}
which clearly leads to
\begin{equation}
\label{PRCVGRN3}
\sum_{n=1}^\infty \Bigl(\frac{\cR_n}{P_{n-1}} \Bigr)^2 <+\infty
\hspace{1cm}\text{a.s.}
\end{equation}
Therefore, we obtain from \eqref{PRCVGHN3} and \eqref{PRCVGRN3} that the two convergences in
\eqref{PRQSLHRN1} hold true, which immediately implies \eqref{QSLDELTAN}. Hereafter, one can notice 
from \eqref{DIFDELTA} that
\begin{equation}
\label{PRQSLFIN}
\sum_{k=1}^n \bigl(  \vtkh - \vta \bigr)^2=\sum_{k=1}^n \Delta_k^2 +
\sum_{k=1}^n \delta_k^2 \bigl(  \theta_k - \ta \bigr)^2 +  2\sum_{k=1}^n \delta_k \Delta_k \bigl(  \theta_k - \ta \bigr). 
\end{equation}
Hence, in order to prove \eqref{QSL1}, it is only necessary to show that
\begin{equation*}
\label{PRQSLFIN1}
\lim_{n \rightarrow \infty} \frac{1}{\log n} \sum_{k=1}^n \delta_k^2 \bigl(  \theta_k - \ta \bigr)^2=0
\hspace{1cm}\text{a.s.}
\end{equation*}
and to make use of the Cauchy-Schwarz inequality. Denote
$$
\Lambda_n=\sum_{k=1}^n \bigl( \theta_k - \ta  \bigr)^2.
$$
We have from \eqref{QSLRMStepa} that as soon as $f(\ta)>0$,
\begin{equation}
\label{PRQSLFIN2}
\lim_{n \rightarrow \infty} \frac{1}{n^{1-a}} \Lambda_n= \frac{a_1\alpha(1-\alpha)}{2(1-a)f(\ta)}
\hspace{1cm}\text{a.s.}
\end{equation}
Furthermore, we obtain from a simple Abel transform that
\begin{equation}
\label{ABELDELTAN}
\sum_{k=1}^{n}  \delta_k^2 \bigl( \theta_k - \ta  \bigr)^2=\delta_n^2 \Lambda_n + \sum_{k=1}^{n-1}(\delta_k^2 - \delta_{k+1}^2 )\Lambda_k.
\end{equation}
We obtain from \eqref{DEFDELTAN} that
\begin{equation}
\label{CVGdelta}
\lim_{n \rightarrow \infty} n^{1-a} \delta_n= \frac{b_1 C_\alpha}{a_1f(\ta)}.
\end{equation}
Then, we deduce from \eqref{PRQSLFIN2} that
$$
\lim_{n \rightarrow \infty}  n^{1-a} \delta_n^2 \Lambda_n= \frac{b_1^2 C_\alpha^2\alpha(1-\alpha)}{2 a_1(1-a)f^3(\ta)}
\hspace{1cm}\text{a.s.}
$$
which implies that
\begin{equation}
\label{PRQSLFIN3}
\lim_{n \rightarrow \infty}  \frac{1}{\log n} \delta_n^2 \Lambda_n=0
\hspace{1cm}\text{a.s.}
\end{equation}
In addition, we also have from \eqref{DEFDELTAN} that
$$
\lim_{n \rightarrow \infty} n^{3-2a} \bigl( \delta_n^2 - \delta_{n+1}^2 \bigr)= 2(1-a) \Bigl(\frac{b_1 C_\alpha}{a_1 f(\ta)}\Bigr)^2.
$$
It clearly ensures via \eqref{PRQSLFIN2} that
\begin{equation}
\label{PRQSLFIN4}
\lim_{n \rightarrow \infty}  \frac{1}{\log n} \sum_{k=1}^{n-1}(\delta_k^2 - \delta_{k+1}^2 )\Lambda_k=0
\hspace{1cm}\text{a.s.}
\end{equation}
Then, it follows from \eqref{ABELDELTAN} together with \eqref{PRQSLFIN3} and \eqref{PRQSLFIN4} that
\begin{equation}
\label{PRQSLFIN5}
\lim_{n \rightarrow \infty}  \frac{1}{\log n} \sum_{k=1}^{n}  \delta_k^2 \bigl( \theta_k - \ta  \bigr)^2=0
\hspace{1cm}\text{a.s.}
\end{equation}
Consequently, we obtain from \eqref{QSLDELTAN} together with \eqref{PRQSLFIN}, \eqref{PRQSLFIN5} and the Cauchy-Schwarz inequality that
$(\vtnh)$ satisfies the QSL
\begin{equation*}
\lim_{n \rightarrow \infty} \frac{1}{\log n} \sum_{k=1}^n \bigl(  \vtkh - \vta \bigr)^2= \Bigl(\frac{b_1^2}{2b_1-1} \Bigr)\tau^2_{\alpha}(\theta_\alpha)
\hspace{1cm}\text{a.s.}
\end{equation*}

$\bullet$
The proof of the QSL for the convexified estimator $(\vtnt)$ is much more easier. We infer from \eqref{DECMART}, \eqref{DEFADDT}, \eqref{DECPN} and the identity 
$L_\alpha(\ta) =\vta$ that for all $n \geq 1$,
\begin{eqnarray}
\vtnpt  -\vta & = & \frac{1}{P_{n}}\Bigl( \widetilde{\vartheta}_1+N_{n+1} + L_{n+1} -P_n \vta\Bigr), \nonumber \\
& = & \frac{1}{P_{n}}\Bigl(N_{n+1} + \widetilde{R}_{n+1} \Bigr)
\label{PRQSLT1}
\end{eqnarray}
where
\begin{equation}
\label{RMDT1}
\widetilde{R}_{n+1}=  \widetilde{\vartheta}_1 -\vta + \sum_{k=1}^{n} b_k P_k \bigl(L_\alpha(\theta_k) - L_\alpha(\ta) \bigr).
\end{equation}
It follows from Lemma \ref{L-MART1} that
\begin{equation}
\label{PRQSLT2}
\lim_{n \rightarrow \infty} \frac{1}{\log n} \sum_{k=1}^n \Bigl(  \frac{N_k}{P_{k-1}} \Bigr)^2= \Bigl(\frac{b_1^2}{2b_1-1} \Bigr)\tau^2_{\alpha}(\theta_\alpha)
\hspace{1cm}\text{a.s.}
\end{equation}
Hence, in order to prove \eqref{QSL1}, it is only necessary to show that
\begin{equation}
\label{PRQSLT3}
\lim_{n \rightarrow \infty} \frac{1}{\log n} \sum_{k=1}^n \Bigl(  \frac{\widetilde{R}_k}{P_{k-1}} \Bigr)^2 =0
\hspace{1cm}\text{a.s.}
\end{equation}
We shall prove the stronger result
\begin{equation}
\label{SRMDT}
 \sum_{n=1}^\infty \Bigl(  \frac{\widetilde{R}_k}{P_{k-1}} \Bigr)^2 < + \infty
\hspace{1cm}\text{a.s.}
\end{equation}
We obtain from \eqref{TAYLORL} and \eqref{RMDT1} that for all $n \geq 1$,
\begin{equation}
\label{RMDTT}
\widetilde{R}_{n+1}^2 \leq  2\bigl(\widetilde{\vartheta}_1 - \ta \bigr)^2+M_f
\Bigl( \sum_{k=1}^{n} b_k P_k \bigl( \theta_k - \ta  \bigr)^2 \Bigr)^2
\end{equation}
where
$$
M_f=\frac{||f||_\infty^2}{2(1- \alpha)^2}.
$$
As before, we obtain from a simple Abel transform that
\begin{equation}
\label{ABELT}
\sum_{k=1}^{n} b_k P_k \bigl( \theta_k - \ta  \bigr)^2=b_nP_n \Lambda_n + \sum_{k=1}^{n-1}(b_kP_k - b_{k+1}P_{k+1} )\Lambda_k.
\end{equation}
It is easy to see that
$$
b_nP_n - b_{n+1}P_{n+1}=b_nP_n \Bigl(\frac{1-b_1}{n+1-b_1}\Bigr).
$$
It implies that
$
0<n(b_nP_n - b_{n+1}P_{n+1})<(1-b_1)b_nP_n.
$
Hence, as $1/2<a<1$ and $1/2<b_1<1$, we find from \eqref{CVGPN} and \eqref{PRQSLFIN2} that
\eqref{SRMDT} holds true.
Consequently, we deduce from \eqref{PRQSLT1} together with \eqref{PRQSLT2} and \eqref{PRQSLT3} that
\begin{equation*}
\lim_{n \rightarrow \infty} \frac{1}{\log n} \sum_{k=1}^n \bigl(  \vtkt - \vta \bigr)^2= \Bigl(\frac{b_1^2}{2b_1-1} \Bigr)\tau^2_{\alpha}(\theta_\alpha)
\hspace{1cm}\text{a.s.}
\end{equation*}
which is exactly the QSL given by \eqref{QSL1}. 

$\bullet$
It only remains to establish the LIL for our estimates $(\vtnh)$ and $(\vtnt)$.
We start by proving the LIL for the sequence $(\Delta_n)$. 
We immediately obtain from \eqref{DECALGON3} that 
\begin{equation}
\label{PRLILH1}
\left(\frac{n}{2 \log \log n} \right)^{1/2}\!\!\Delta_n=
\left(\frac{n}{2 \log \log n} \right)^{1/2}\Bigl(\frac{\cM_n +\cH_n+\cR_n }{P_{n-1}}\Bigr).
\end{equation}
We already saw in Lemma \ref{L-MART1} that the martingale $(\cM_n)$ satisfies the LIL given by 
\eqref{LILMARTN1}. In addition, it is easy to see from \eqref{CVGPN}, \eqref{PRCVGHN2} and \eqref{PRCVGRN2} that
$$
\lim_{n \rightarrow \infty} n \Bigl( \frac{\cH_n }{P_{n-1}}\Bigr)^2=0 \hspace{1cm}\text{and}\hspace{1cm}
\lim_{n \rightarrow \infty} n \Bigl( \frac{\cR_n }{P_{n-1}}\Bigr)^2=0 
\hspace{1cm}\text{a.s.}
$$
It clearly implies that
\begin{equation*}
\lim_{n \rightarrow \infty} \left(\frac{n}{2 \log \log n} \right)^{1/2} \Bigl( \frac{\cH_n }{P_{n-1}}\Bigr) =0
\hspace{1cm}\text{a.s.}
\end{equation*}
and
\begin{equation*}
\lim_{n \rightarrow \infty} \left(\frac{n}{2 \log \log n} \right)^{1/2} \!\!\Bigl( \frac{\cR_n }{P_{n-1}}\Bigr) =0
\hspace{1cm}\text{a.s.}
\end{equation*}
Therefore, we deduce from \eqref{LILMARTN1} and \eqref{PRLILH1} that $(\Delta_n)$ satisfies the LIL 
\begin{eqnarray}  
\limsup_{n \rightarrow \infty} \left(\frac{n}{2  \log \log n} \right)^{1/2}
 \!\!\Delta_n
&=& -  \liminf_{n \rightarrow \infty} \left(\frac{n}{2  \log \log n} \right)^{1/2}
 \!\!\Delta_n  \notag \\
&=& \left(  \frac{b_1^2}{2b_1 -1}  \right)^{1/2} \tau_{\alpha}(\theta_\alpha)
\hspace{1cm}\text{a.s.}
\label{LILDELTAN1}
\end{eqnarray}
Hereafter, one can observe from \eqref{DIFDELTA} that
\begin{equation}
\label{PRLILH2}
\left(\frac{n}{2  \log \log n} \right)^{1/2}
\!\!\!\bigl( \vtnh - \vta \bigr)=\left(\frac{n}{2  \log \log n} \right)^{1/2}
 \!\!\!\Delta_n+\left(\frac{n}{2  \log \log n} \right)^{1/2}
 \!\!\!\delta_n \bigl( \tn- \ta \bigr).
\end{equation}
It follows from \eqref{LILRMStepa} and \eqref{CVGdelta} that
$$
\left(\frac{n}{2  \log \log n} \right)
 \!\!\!\ \bigl| \delta_n ( \tn- \ta) \bigr|^2=O\left(\frac{\log n}{n^{1-a} \log \log n} \right)
\hspace{1cm}\text{a.s.}
$$
which clearly leads to
\begin{equation}
\label{PRLILH3}
\lim_{n \rightarrow \infty} \left(\frac{n}{2 \log \log n} \right)^{1/2} \!\!\delta_n \bigl( \tn- \ta \bigr) =0
\hspace{1cm}\text{a.s.}
\end{equation}
Consequently, we obtain \eqref{LIL1}  from \eqref{LILDELTAN1}, \eqref{PRLILH2} and \eqref{PRLILH3}. The proof for the convexified estimator $(\vtnt)$ is
straightforward. We obtain from \eqref{PRQSLT1} that 
\begin{equation}
\label{PRLILT1}
\left(\frac{n}{2 \log \log n} \right)^{1/2}\bigl( \vtnt - \vta \bigr)=
\left(\frac{n}{2 \log \log n} \right)^{1/2}\Bigl(\frac{N_n +\widetilde{R}_n }{P_{n-1}}\Bigr).
\end{equation}
We already saw in Lemma \ref{L-MART1} that the martingale $(N_n)$ satisfies the LIL given by 
\eqref{LILMARTN1}. In addition, it is easy to see from \eqref{CVGPN}, \eqref{RMDTT} and \eqref{ABELT} that
$$
\lim_{n \rightarrow \infty} n \Bigl( \frac{\widetilde{R}_n }{P_{n-1}}\Bigr)^2=0 \hspace{1cm}\text{a.s.}
$$
which clearly implies 
\begin{equation}
\label{PRLILT2}
\lim_{n \rightarrow \infty} \left(\frac{n}{2 \log \log n} \right)^{1/2} \Bigl( \frac{\widetilde{R}_n }{P_{n-1}}\Bigr) =0
\hspace{1cm}\text{a.s.}
\end{equation}
Finally, we deduce \eqref{LIL1} from \eqref{LILMARTN1}, \eqref{PRLILT1} and \eqref{PRLILT2},
which completes the proof of Theorem \ref{T-LILQSEQUAL1}.
\demend
%%%%%%%%%%%%%%%%%%%%%%%%%%%%%%%%%%%%%%%%%%%%%%%%%%%%%%%%%%%%%%%%%%%%%%%%%%%%%%%%%%%%%%%%%%%%%%%%%%

\subsection{The slow step size case.}
In order to prove Theorem \ref{T-LILQSLESS1}, it is necessary to establish the following QSL and LIL
for the martingales $(\cM_n)$ and $(N_n)$.

\begin{lem}
\label{L-MARTLESS1}
Assume that the step sequences $(a_n)$ and $(b_n)$ are given by
\begin{equation*}
a_n=\frac{a_1}{n^a}
\hspace{1.5cm}\text{and}\hspace{1.5cm}
b_n=\frac{b_1}{n^b}
\end{equation*}
where $a_1>0$,  $b_1>0$ and $1/2<a<b<1$. Then, $(\cM_n)$ and $(N_n)$ share the same QSL
\begin{equation}
\label{QSLMARTMLESS1}
\lim_{n \rightarrow \infty} \frac{1}{n^{1-b}} \sum_{k=1}^n \Bigl(  \frac{\cM_k}{P_{k-1}} \Bigr)^2= \Bigl(\frac{b_1}{2(1-b)} \Bigr)\tau^2_{\alpha}(\theta_\alpha)
\hspace{1cm}\text{a.s.}
\end{equation}
%\begin{equation}
%\label{QSLMARTN1}
%\lim_{n \rightarrow \infty} \frac{1}{\log n} \sum_{k=1}^n \Bigl(  \frac{N_k}{P_{k-1}} \Bigr)^2= \Bigl(\frac{b_1^2}{2b_1-1} \Bigr)\tau^2_{\alpha}(\theta_\alpha)
%\hspace{1cm}\text{a.s.}
%\end{equation}
In addition, they also share the same LIL
\begin{eqnarray}  
\limsup_{n \rightarrow \infty} \left(\frac{n^b}{2 (1-b) \log n} \right)^{1/2}
 \!\!\!\Bigl(\frac{\cM_n}{P_{n-1}}\Bigr)
&=& - \liminf_{n \rightarrow \infty} \left(\frac{n^b}{2 (1-b) \log n} \right)^{1/2}
 \!\!\!\Bigl(\frac{\cM_n}{P_{n-1}}\Bigr)  \notag \\
&=& \left(  \frac{b_1}{2}  \right)^{1/2} \tau_{\alpha}(\theta_\alpha)
\hspace{1cm}\text{a.s.}
\label{LILMARTMLESS1}
\end{eqnarray}
%\begin{eqnarray}  
%\limsup_{n \rightarrow \infty} \left(\frac{n}{2  \log \log n} \right)^{1/2}
% \!\!\!\Bigl(\frac{N_n}{P_{n-1}}\Bigr)
%&=& - \left(\frac{n}{2  \log \log n} \right)^{1/2}
% \!\!\!\Bigl(\frac{N_n}{P_{n-1}}\Bigr)  \notag \\
%&=& \left(  \frac{b_1^2}{2b_1 -1}  \right)^{1/2} \tau_{\alpha}(\theta_\alpha)
%\hspace{1cm}\text{a.s.}
%\label{LILMARTN1}
%\end{eqnarray}
%Moreover, $(N_n)$ shares exactly the same quadratic strong law and the same law of the iterated logarithm than $(\cM_n)$.
\end{lem}

\begin{proof}
We recall that the martingale $(\cM_n)$ and its predictable quadratic variation are given by
$$
\cM_{n+1}=\sum_{k=1}^{n} b_k P_k W_{k+1}
\hspace{1cm}\text{and}\hspace{1cm}
\ccMc_{n+1} = \sum_{k=1}^{n} b_k^2 P_k^2 \tau^2_k(\theta_k)
$$
where, thanks to \eqref{ASCVGNU},
$$
\lim_{n \rightarrow \infty} \tau_n^2(\tn) = \tau_\alpha^2(\ta) 
\hspace{1cm}\text{a.s}
$$
It is not hard to see via a comparison series integral together with convergence \eqref{CVGPNZETA} in Lemma \ref{L-CVGPN} that
\begin{equation}
\label{COMPSI}
\lim_{n \rightarrow \infty} \frac{1}{b_nP_n^2}\sum_{k=1}^n b_k^2 P_k^2 = \frac{1}{2}.
\end{equation}
Hence, we deduce from \eqref{COMPSI} and Toeplitz's lemma that
\begin{equation}
\label{CVGIPCMNLESS1}
\lim_{n \rightarrow \infty} \frac{1}{b_{n}P_{n}^2} \ccMc_{n+1} =  \frac{\tau^2_\alpha(\ta)}{2}
\hspace{1cm}\text{a.s.}
\end{equation}
Denote by $f_n$ the explosion coefficient associated with the martingale $(\cM_n)$,
$$
f_n = \frac{\ccMc_n - \langle \cM \rangle_{n-1}}{\ccMc_n}
$$
It follows from the very definition of
$P_n$ given by \eqref{DEFPN} together with \eqref{CVGIPCMNLESS1} that
\begin{equation}
\label{CVGFNLESS1}
\lim_{n \rightarrow \infty} n^bf_n=2b_1\hspace{1cm}\text{a.s.}
\end{equation}
It means that $f_n$ converges to zero almost surely at rate $n^b$ where $1/2<b<1$. 
Furthermore, the random variable $X$ has a moment of order $>2$. It implies that for some real number $p>2$,
$$
\sup_{n \geq 0} \dE[ | W_{n+1} |^p | \cF_n] < \infty \hspace{1cm}\text{a.s.}
$$
Consequently, we deduce from the QSL for martingales given in theorem 3 of \cite{bercu2004} that
\begin{equation}
\label{QSLCMNLESS1}
\lim_{n \rightarrow \infty} \frac{1}{\log \ccMc_n}   \sum_{k=1}^n f_k  \frac{\cM_k^2}{\langle \cM \rangle _k} 
 = 1
\hspace{1cm}\text{a.s.}
\end{equation}
Therefore, we obtain from \eqref{CVGPNLESS1} and \eqref{CVGIPCMNLESS1} together with \eqref{CVGFNLESS1} and \eqref{QSLCMNLESS1} that
\begin{equation}
\label{QSLCMNLESS2}
\lim_{n \rightarrow \infty} \frac{1}{n^{1-b}}   \sum_{k=1}^n   \frac{\cM_k^2}{P_{k-1}^2 }
 = \Bigl(\frac{b_1}{2(1-b)} \Bigr)\tau^2_{\alpha}(\theta_\alpha)
\hspace{1cm}\text{a.s.}
\end{equation}
Hereafter, we focus our attention on the proof of the LIL given 
by \eqref{LILMARTMLESS1}. Since $b>1/2$, we obtain from \eqref{CVGFNLESS1} that the explosion coefficient $f_n$ satisfies
$$
\sum_{n=1}^\infty f_n^{p/2} < +\infty
\hspace{1cm}\text{a.s.}
$$
Therefore, we deduce from the LIL for martingales \cite{Stout1970}, see also corollary 6.4.25 in \cite{Duflo1997} that
\begin{eqnarray}   
\limsup_{n \rightarrow \infty} \left(\frac{1}{2 \ccMc_n \log \log \ccMc_n} \right)^{1/2}
 \!\!\!\cM_n
&=& - \liminf_{n \rightarrow \infty}
\left(\frac{1}{2 \ccMc_n \log \log \ccMc_n} \right)^{1/2}
 \!\!\!\cM_n  \notag \\
&=& 1
\hspace{1cm}\text{a.s.}
\label{LILCMNLESS1}
\end{eqnarray}
Hence, we find from \eqref{CVGPNLESS1}, \eqref{CVGIPCMNLESS1} and \eqref{LILCMNLESS1} that
\begin{eqnarray*}   
\limsup_{n \rightarrow \infty} \left(\frac{n^b}{2 (1-b)\log n} \right)^{1/2}
 \!\!\!\Bigl(\frac{\cM_n}{P_{n-1}}\Bigr)
&=& - \liminf_{n \rightarrow \infty} \left(\frac{n^b}{ 2 (1-b) \log n} \right)^{1/2}
 \!\!\!\Bigl(\frac{\cM_n}{P_{n-1}}\Bigr)  \notag \\
&=& \left(  \frac{b_1}{2}  \right)^{1/2} \tau_{\alpha}(\theta_\alpha)
\hspace{1cm}\text{a.s.}
%\label{LILCMN2}
\end{eqnarray*}
The proof for the martingale $(N_n)$ is left to the reader inasmuch as it follows exactly the same lines than those for
the martingale $(\cM_n)$.
\end{proof}

\noindent{\bf Proof of Theorem \ref{T-LILQSLESS1}.} 
%%%%%%%%%%%%%%%%%%%%%%%%%%%%%%%%%%%%%%%%%%%%%%%%%%%%%%%%%%%%%%%%%%%%%%%%%%%%%%%%%%%%%%%%%%%%%%%%%%
We shall proceed as in the proof of Theorem \ref{T-LILQSEQUAL1}. We already saw from \eqref{QSLMARTMLESS1} that 
\begin{equation*}
\lim_{n \rightarrow \infty} \frac{1}{n^{1-b}} \sum_{k=1}^n \Bigl(  \frac{\cM_k}{P_{k-1}} \Bigr)^2= \Bigl(\frac{b_1}{2(1-b)} \Bigr)\tau^2_{\alpha}(\theta_\alpha)
\hspace{1cm}\text{a.s.}
\end{equation*}
Our goal is to prove that the sequence $(\Delta_n)$ given by \eqref{DIFDELTA} satisfies the QSL
\begin{equation}
\label{QSLDELTANLESS1}
\lim_{n \rightarrow \infty} \frac{1}{n^{1-b}} \sum_{k=1}^n \Delta^2_k= \Bigl(\frac{b_1}{2(1-b)} \Bigr)\tau^2_{\alpha}(\theta_\alpha)
\hspace{1cm}\text{a.s.}
\end{equation}
On the one hand, we have from \eqref{DEFMARTNEW} and \eqref{MAJLILStepa} that
\begin{equation*}
| \cH_{n+1} | = O \left( \sum_{k=1}^n \frac{b_kP_k | \nu_{k+1}| \sqrt{\log k}}{k^{a/2}} \right)
\hspace{1cm}\text{a.s.}
\end{equation*}
In addition, one can easily check from \eqref{DEFDELTAN} and \eqref{DEFNUNN} that
$$
\lim_{n \rightarrow \infty}  n^{1-a} \nu_{n+1}= \frac{(b-a)C_\alpha}{a_1f(\ta)}.
$$
Hence, we obtain from convergence \eqref{CVGPNZETA} in Lemma \ref{L-CVGPN} together with a comparison series integral 
as previously done in the proof of Theorem \ref{T-LILQSEQUAL1} that
\begin{equation}
\label{PRCVGHNLESS1}
| \cH_{n+1} | = O \left( \sum_{k=1}^n \frac{P_k \sqrt{\log k}}{k^{1+b-a/2}} \right)=O \left( \frac{P_n\sqrt{\log n}}{n^{1-a/2}} \right)
\hspace{1cm}\text{a.s.}
\end{equation}
Consequently, we deduce from \eqref{PRCVGHNLESS1} that
\begin{equation}
\label{PRCVGHNLESS2}
\sum_{n=1}^\infty \Bigl(\frac{\cH_n}{P_{n-1}} \Bigr)^2 <+\infty
\hspace{1cm}\text{a.s.}
\end{equation}
On the other hand, we already saw from \eqref{PRCVGRN1} that
\begin{equation*}
| \cR_{n+1} | = O \left( \sum_{k=1}^n b_kP_k\bigl(\theta_k - \ta\bigr)^2 \right)
\hspace{1cm}\text{a.s.}
\end{equation*}
which implies that
\begin{equation}
\label{PRCVGRNLESS1}
| \cR_{n+1} | = O \left( \sum_{k=1}^n \frac{P_k\log k }{k^{a+b}} \right) = O \left( \frac{P_n \log n }{n^{a}} \right)
\hspace{1cm}\text{a.s.}
\end{equation}
Then, as $a>1/2$, we find from \eqref{PRCVGRNLESS1}
\begin{equation}
\label{PRCVGRNLESS2}
\sum_{n=1}^\infty \Bigl(\frac{\cR_n}{P_{n-1}} \Bigr)^2 <+\infty
\hspace{1cm}\text{a.s.}
\end{equation}
Therefore, we obtain from \eqref{PRCVGHNLESS2} and \eqref{PRCVGRNLESS2} that the QSL 
\eqref{QSLDELTANLESS1} holds true. In order to prove \eqref{QSL3}, it only remains to show via \eqref{PRQSLFIN} that
\begin{equation}
\label{PRQSLFINLESS1}
\lim_{n \rightarrow \infty} \frac{1}{n^{1-b}} \sum_{k=1}^n \delta_k^2 \bigl(  \theta_k - \ta \bigr)^2=0
\hspace{1cm}\text{a.s.}
\end{equation}
We recall from \eqref{ABELDELTAN} that
\begin{equation*}
\sum_{k=1}^{n}  \delta_k^2 \bigl( \theta_k - \ta  \bigr)^2=\delta_n^2 \Lambda_n + \sum_{k=1}^{n-1}(\delta_k^2 - \delta_{k+1}^2 )\Lambda_k.
\end{equation*}
We obtain from \eqref{DEFDELTAN} that
\begin{equation}
\label{CVGdeltaLESS}
\lim_{n \rightarrow \infty} n^{b-a} \delta_n= \frac{b_1 C_\alpha}{a_1 f(\ta)}.
\end{equation}
Then, it follows from \eqref{PRQSLFIN2} and \eqref{CVGdeltaLESS} that
\begin{equation}
\label{PRQSLFINLESS2}
\lim_{n \rightarrow \infty}  \frac{1}{n^{1+a-2b}} \delta_n^2 \Lambda_n= \frac{b_1^2 C_\alpha^2\alpha(1-\alpha)}{2 a_1(1-a)f^3(\ta)}
\hspace{1cm}\text{a.s.}
\end{equation}
Consequently, as $a<b$, we deduce from \eqref{PRQSLFINLESS2} that
\begin{equation}
\label{PRQSLFINLESS3}
\lim_{n \rightarrow \infty}  \frac{1}{n^{1-b}} \delta_n^2 \Lambda_n=0
\hspace{1cm}\text{a.s.}
\end{equation}
By the same token, we also find from \eqref{PRQSLFIN2} and \eqref{CVGdeltaLESS} that
\begin{equation}
\label{PRQSLFINLESS4}
\lim_{n \rightarrow \infty}  \frac{1}{n^{1-b}} \sum_{k=1}^{n-1}(\delta_k^2 - \delta_{k+1}^2 )\Lambda_k=0
\hspace{1cm}\text{a.s.}
\end{equation}
Then, we clearly obtain from \eqref{PRQSLFINLESS3} and \eqref{PRQSLFINLESS4} that
convergence \eqref{PRQSLFINLESS1} holds true. As before, the proof of the QSL for the convexified estimator $(\vtnt)$ is much more easier
and left to the reader.
%%%%%%%%%%%%%%%%%%%%%%%%%%%%%%%%%%%%%%%%%%%%%%%%%%%%%%%%%%%%%%%%%%%%%%%%%%%%%%%%%%%%%%%%%%%%%%%%%%
We now focus our attention on the LIL for our estimates $(\vtnh)$ and $(\vtnt)$.
We start by proving the LIL for the sequence $(\Delta_n)$ given by \eqref{DIFDELTA}. 
We immediately obtain from \eqref{DECALGON3} that 
\begin{equation}
\label{PRLILHLESS1}
\left(\frac{n^b}{2(1-b) \log n} \right)^{1/2}\!\!\Delta_n=
\left(\frac{n^b}{2(1-b) \log n} \right)^{1/2}\Bigl(\frac{\cM_n +\cH_n+\cR_n }{P_{n-1}}\Bigr).
\end{equation}
We already saw in Lemma \ref{L-MARTLESS1} that the martingale $(\cM_n)$ satisfies the LIL given by 
\eqref{LILMARTMLESS1}. In addition, as $b<1<2a$, we get from \eqref{PRCVGHNLESS1} and \eqref{PRCVGRNLESS1} that
$$
\lim_{n \rightarrow \infty} n^b \Bigl( \frac{\cH_n }{P_{n-1}}\Bigr)^2=0 \hspace{1cm}\text{and}\hspace{1cm}
\lim_{n \rightarrow \infty} n^b \Bigl( \frac{\cR_n }{P_{n-1}}\Bigr)^2=0 
\hspace{1cm}\text{a.s.}
$$
which clearly ensures that
\begin{equation*}
\lim_{n \rightarrow \infty} \left(\frac{n^b}{2(1-b) \log n} \right)^{1/2} \Bigl( \frac{\cH_n }{P_{n-1}}\Bigr) =0
\hspace{1cm}\text{a.s.}
\end{equation*}
and
\begin{equation*}
\lim_{n \rightarrow \infty} \left(\frac{n^b }{2(1-b) \log n} \right)^{1/2} \!\!\Bigl( \frac{\cR_n }{P_{n-1}}\Bigr) =0
\hspace{1cm}\text{a.s.}
\end{equation*}
Consequently, we find from \eqref{LILMARTMLESS1} and \eqref{PRLILHLESS1} that $(\Delta_n)$ satisfies the LIL
\begin{eqnarray}  
\limsup_{n \rightarrow \infty} \left(\frac{n^b}{2(1-b)  \log n} \right)^{1/2}
 \!\!\Delta_n
&=& -  \liminf_{n \rightarrow \infty} \left(\frac{n^b}{2 (1-b) \log n} \right)^{1/2}
 \!\!\Delta_n  \notag \\
&=& \left(  \frac{b_1}{2}  \right)^{1/2} \tau_{\alpha}(\theta_\alpha)
\hspace{1cm}\text{a.s.}
\label{LILDELTANLESS1}
\end{eqnarray}
Hereafter, we clearly have from \eqref{DIFDELTA} that
\begin{equation}
\label{PRLILHLESS2}
\left(\frac{n^b}{2 (1-b) \log n} \right)^{1/2}
\!\!\bigl( \vtnh - \vta \bigr)=\left(\frac{n^b}{2 (1-b) \log n} \right)^{1/2}
\!\!\bigl( \Delta_n + \delta_n \bigl( \tn- \ta \bigr) \bigr).
\end{equation}
It follows from \eqref{LILRMStepa} and \eqref{CVGdeltaLESS} that
$$
\left(\frac{n^b}{2 (1-b)  \log n} \right)
 \!\!\!\ \bigl| \delta_n ( \tn- \ta) \bigr|^2=O\left( \frac{1}{n^{b-a}} \right)
\hspace{1cm}\text{a.s.}
$$
Since $a<b$, it clearly implies that
\begin{equation}
\label{PRLILHLESS3}
\lim_{n \rightarrow \infty} \left(\frac{n^b}{2 (1-b)\log n} \right)^{1/2} \!\!\delta_n \bigl( \tn- \ta \bigr) =0
\hspace{1cm}\text{a.s.}
\end{equation}
Therefore, we obtain \eqref{LIL3}  from \eqref{LILDELTANLESS1}, \eqref{PRLILHLESS2} and \eqref{PRLILHLESS3}. 
The proof of the LIL for the convexified estimator $(\vtnt)$ is
straightforward and left to the reader, which achieves the proof of Theorem \ref{T-LILQSLESS1}.
\demend

%%%%%%%%%%%%%%%%%%%%%%%%%%%%%%%%%%%%%%%%%%%%%%%%%%%%%%%%%%%%%%%%%%%%%%%%%%%%%%%%%%%%%%%%%%%%%%%%%%
\vspace{-2ex}
%%%%%%%%%%%%%%%%%%%%%%%%%%%%%%%%%%%%%%%%%%%%%%%%%%%%%%%%%%%%%%%%%%%%%%%%%%%%%%%%%%%%%%%%%%%%%%%%%%

\section{Proofs of the asymptotic normality results}
\label{S-PRAN}

The proof of Theorem \ref{T-AN} relies on the central limit theorem for the two-time-scale stochastic 
algorithm given in Theorem 1 of Mokkadem and Pelletier \cite{MokkademPelletier2006}. It is a sophisticated application of
this result for the standard estimator $(\vtnh)$, while it is a direct application for the convexified estimator $(\vtnt)$.
\ \vspace{1ex}\\
\noindent{\bf Proof of Theorem \ref{T-AN}.} 
%%%%%%%%%%%%%%%%%%%%%%%%%%%%%%%%%%%%%%%%%%%%%%%%%%%%%%%%%%%%%%%%%%%%%%%%%%%%%%%%%%%%%%%%%%%%%%%%%%
We start with the proof for the standard estimator $\vtnh$. As it was previously done in Section \ref{S-PRASCVG}, our strategy is
first to establish the joint asymptotic normality for the couple $(\theta_n,\Delta_n)$ where $\Delta_n$ is given by \eqref{DIFDELTA}, 
and then to deduce the joint asymptotic normality for the couple $(\theta_n,\vtnh)$.
We have from \eqref{TTSALGO1} together with \eqref{DECALGON2} that
for all $n \geq 1$, 
\begin{equation}
\label{PRAN1}
\left \{
\begin{aligned}
&\tnp  = \tn+a_n \cX_{n+1}\vspace{1ex}\\
&\Delta_{n+1}  =\Delta_n+b_n \cY_{n+1}
\end{aligned}
\right.
\end{equation}
where
\begin{equation*}
\left \{
\begin{aligned}
&\cX_{n+1}  =  f(\tn,\Delta_n)+\psi_n^{(\theta)}+ \cV_{n+1}\\
&\cY_{n+1} =  g(\tn,\Delta_n)+\psi_n^{(\Delta)}+ \cW_{n+1}\\
\end{aligned}
\right.
\end{equation*}
with $f(\theta, \Delta)=\alpha - F(\theta)$, $\psi_n^{(\theta)}=0$, $\cV_{n+1}=F(\theta_n)-\rI_{\{X_{n+1} \leq \tn \}}$ and
$g(\theta, \Delta)=- \Delta$, 
$$
\psi_n^{(\Delta)}=R_\alpha(\tn) +\frac{a_n}{b_n}\delta_{n+1} G_\alpha(\theta_n) +\nu_{n+1} (\tn - \ta), 
$$
$$
\cW_{n+1}= \veps_{n+1}+\frac{a_n}{b_n}\delta_{n+1}V_{n+1}.
$$
By denoting $\Delta_\alpha=0$, we clearly have $f(\ta, \Delta_\alpha)=0$ and $g(\ta, \Delta_\alpha)=0$. To be more precise 
\begin{equation*}
\begin{pmatrix}
f(\theta, \Delta_\alpha) \\
g(\theta, \Delta_\alpha)
\end{pmatrix}
=
\begin{pmatrix}
-f^\prime(\ta) & 0 \\
0 &-1
\end{pmatrix}
\begin{pmatrix}
\theta - \ta\\
\Delta - \Delta_\alpha
\end{pmatrix}
+
\begin{pmatrix}
O\bigl( ||\theta - \ta||^2 \bigr)\\
0
\end{pmatrix}
.
\end{equation*}
On the one hand, it follows from the conjunction of \eqref{TAYLORH}, \eqref{DEFREMAINDERS} and \eqref{DEFDELTAN} that
$$
\psi_n^{(\Delta)}=r_n^{(\Delta)} + O\bigl( ||\theta_n - \ta||^2 \bigr)
$$
where $r_n^{(\Delta)}=nu_{n+1} (\tn - \ta)$. On the other hand, we infer from \eqref{LILSUPRMStepa} and \eqref{DEFNUNN} that
$$
\bigl| r_n^{(\Delta)} \bigr| = O\Bigl( \frac{\sqrt{n^a \log n}}{n}\Bigr)=o\bigl(\sqrt{b_n}\bigr)
\hspace{1cm}\text{a.s.}
$$
Furthermore, $\dE[\cV_{n+1} | \cF_n]=0$,  $\dE[\cW_{n+1} | \cF_n]=0$, and we already saw in Sections \ref{S-MA} and \ref{S-PRASCVG} that 
$\dE[\cV_{n+1}^2 | \cF_n]=F(\theta_n)(1-F(\theta_n)$ and $\dE[\cW_{n+1}^2 | \cF_n]=\tau_\alpha^2(\tn)$.
One can also check that 
$$
\dE[\cV_{n+1} \cW_{n+1} | \cF_n]=F(\theta_n)\Bigl(H_\alpha(\theta_n) - \frac{a_n}{b_n}\delta_{n+1}\bigl(1-F(\theta_n)\bigr)\Bigr).
$$
It clearly implies that
\begin{equation*}
\lim_{n \rightarrow \infty} 
\begin{pmatrix}
\dE[\cV_{n+1}^2 | \cF_n] & \dE[\cV_{n+1} \cW_{n+1} | \cF_n] \\
\dE[\cV_{n+1} \cW_{n+1} | \cF_n] & \dE[\cW_{n+1}^2 | \cF_n]
\end{pmatrix}
=
\begin{pmatrix}
\alpha(1-\alpha) & \alpha ( \vta - \ta ) \\
\alpha ( \vta - \ta )  & \tau_\alpha^2(\ta)
\end{pmatrix}
\hspace{0.5cm}\text{a.s.}
\end{equation*}
Consequently, all the conditions of Theorem 1 in \cite{MokkademPelletier2006}
are satisfied with
\begin{equation*}
\Sigma_{\ta}= \frac{\alpha(1-\alpha)}{2f(\ta)}
\end{equation*}
and
\begin{equation*}
\Sigma_{\vta}= \left \{
\begin{array}[c]{ccc}
{\displaystyle \frac{b_1 \tau^2_{\alpha}(\theta_\alpha)}{2b_1 - 1}}  & \text{if} & b=1, \vspace{1ex} \\
{\displaystyle \frac{ \tau^2_{\alpha}(\theta_\alpha)}{2}} & \text{if} & b<1.
\end{array}
\right.
\end{equation*}
Therefore, as $\Delta_\alpha=0$, we obtain from \cite{MokkademPelletier2006} the joint asymptotic normality 
\begin{equation}
\label{PRAN2}
\begin{pmatrix}
\sqrt{n^a} \bigl(\theta_n - \ta\bigr) \vspace{1ex} \\
\sqrt{n^b} \Delta_n  \\
\end{pmatrix} \liml \cN\left(0,  \begin{pmatrix}
\Gamma_{\ta} & 0 \\
0 & \Gamma_{\vta} \\
\end{pmatrix}\right)
\end{equation} 
where $\Gamma_{\ta}=a_1 \Sigma_{\ta}$ and $\Gamma_{\vta}=b_1 \Sigma_{\vta}$. Hereafter, in order to prove the
joint asymptotic normality for the couple $(\theta_n,\vtnh)$,  it is only necessary to show from the
very definition of $\Delta_n$ given in \eqref{DIFDELTA} that
\begin{equation}
\label{PRAN3}
\lim_{n \rightarrow \infty} \sqrt{n^b} \delta_n \bigl(\tn -\ta\bigr)=0
\hspace{1cm}\text{a.s.}
\end{equation}
We already saw from \eqref{CVGdelta} and \eqref{CVGdeltaLESS} that
\begin{equation}
\label{PRAN4}
\lim_{n \rightarrow \infty} n^{b-a} \delta_n= \frac{b_1 C_\alpha}{a_1 f(\ta)}.
\end{equation}
Hence, we deduce from \eqref{LILSUPRMStepa} and \eqref{PRAN4} that
\begin{equation*}
 \sqrt{n^b} \bigl| \delta_n (\tn -\ta)\bigr|=O\Bigl( \frac{\sqrt{n^a \log n}}{\sqrt{n^b}}\Bigr)
\hspace{1cm}\text{a.s.}
\end{equation*}
which ensures that \eqref{PRAN3} holds true. Consequently, \eqref{AN1} clearly follows from 
\eqref{PRAN2} and \eqref{PRAN3} .
%%%%%%%%%%%%%%%%%%%%%%%%%%%%%%%%%%%%%%%%%%%%%%%%%%%%%%%%%%%%%%%%%%%%%%%%%%%%%%%%%%%%%%%%%%%%%%%%%%
The proof for the convexified estimator $(\vtnt)$ is much more easy to handle. We have from \eqref{TTSALGO2} that
for all $n \geq 1$, 
\begin{equation}
\label{PRAN5}
\left \{
\begin{aligned}
&\tnp  = \tn+a_n \cX_{n+1}\vspace{1ex}\\
&\vtnpt  =\vtnt+b_n \cY_{n+1}
\end{aligned}
\right.
\end{equation}
where
\begin{equation*}
\left \{
\begin{aligned}
&\cX_{n+1}  =  f(\tn,\vtnt)+\psi_n^{(\theta)}+ \cV_{n+1}\\
&\cY_{n+1} =  g(\tn,\vtnt)+\psi_n^{(\vartheta)}+ \cW_{n+1}\\
\end{aligned}
\right.
\end{equation*}
with $f(\theta, \vartheta)=\alpha - F(\theta)$, $\psi_n^{(\theta)}=0$, $\cV_{n+1}=F(\theta_n)-\rI_{\{X_{n+1} \leq \tn \}}$ and
$g(\theta, \vartheta)=\vta - \vartheta$, $\psi_n^{(\vartheta)}=L_\alpha(\theta_n) - \vta$, $\cW_{n+1}=Z_{n+1}-L_\alpha(\theta_n)$,
where we recall that $\dE[Z_{n+1} | \cF_n]=L_\alpha(\theta_n)$ with $L_\alpha(\theta)$ given by  \eqref{DEFHL}.
We clearly have $f(\ta, \vta)=0$ and $g(\ta, \vta)=0$. To be more precise, 
\begin{equation*}
\begin{pmatrix}
f(\theta, \vartheta) \\
g(\theta, \vartheta)
\end{pmatrix}
=
\begin{pmatrix}
-f^\prime(\ta) & 0 \\
0 &-1
\end{pmatrix}
\begin{pmatrix}
\theta - \ta\\
\vartheta - \vta
\end{pmatrix}
+
\begin{pmatrix}
O\bigl( ||\theta - \ta||^2 \bigr)\\
0
\end{pmatrix}
.
\end{equation*}
In addition, we deduce from \eqref{TAYLORL} that
$ \psi_n^{(\vartheta)} = L_\alpha(\theta_n) - L_\alpha(\ta) = O\bigl( ||\theta_n - \ta||^2 \bigr)$.
Furthermore, $\dE[\cV_{n+1} | \cF_n]=0$,  $\dE[\cW_{n+1} | \cF_n]=0$, and we already saw in Sections \ref{S-MA} and \ref{S-PRASCVG} that 
$\dE[\cV_{n+1}^2 | \cF_n]=F(\theta_n)(1-F(\theta_n)$ and $\dE[\cW_{n+1}^2 | \cF_n]=\tau_\alpha^2(\tn)$.
One can also verify that $\dE[\cV_{n+1} \cW_{n+1} | \cF_n]=F(\theta_n)\bigl(L_\alpha(\theta_n) - \theta_n\bigr)$.
It clearly implies that
\begin{equation*}
\lim_{n \rightarrow \infty} 
\begin{pmatrix}
\dE[\cV_{n+1}^2 | \cF_n] & \dE[\cV_{n+1} \cW_{n+1} | \cF_n] \\
\dE[\cV_{n+1} \cW_{n+1} | \cF_n] & \dE[\cW_{n+1}^2 | \cF_n]
\end{pmatrix}
=
\begin{pmatrix}
\alpha(1-\alpha) & \alpha ( \vta - \ta ) \\
\alpha ( \vta - \ta )  & \tau_\alpha^2(\ta)
\end{pmatrix}
\hspace{0.5cm}\text{a.s.}
\end{equation*}
Consequently, our two-time-scale stochastic algorithm satisfies all the conditions of Theorem 1 in \cite{MokkademPelletier2006}
where the asymptotic variances $\Sigma_{\ta}$ and $\Sigma_{\vta}$ have been previously defined.
Finally, we obtain the joint asymptotic normality \eqref{AN1} where
$\Gamma_{\ta}=a_1 \Sigma_{\ta}$ and $\Gamma_{\vta}=b_1 \Sigma_{\vta}$, which completes the proof of Theorem \ref{T-AN}.
\demend

%%%%%%%%%%%%%%%%%%%%%%%%%%%%%%%%%%%%%%%%%%%%%%%%%%%%%%%%%%%%%%%%%%%%%%%%%%%%%%%%%%%%%%%%%%%%%%%%%%
\vspace{-2ex}
%%%%%%%%%%%%%%%%%%%%%%%%%%%%%%%%%%%%%%%%%%%%%%%%%%%%%%%%%%%%%%%%%%%%%%%%%%%%%%%%%%%%%%%%%%%%%%%%%%

\section{Numerical experiments on real data}
\label{S-NE}

%%%%%%%%%%%%%%%%%%%%%%%%%%%%%%%%%%%%%%%%%%%%%%%%%%%%%%%%%%%%%%%%%%%%%%%%%%%%%%%%%%%%%%%%%%%%%%%%%%
We briefly illustrate the asymptotic behavior of our two stochastic algorithms  $(\vtnh)$ and $(\vtnt)$  with different tuning of parameters. 
Since we have several elements of variability in the parameters, we have chosen typical setups even if our presentation is not exhaustive.
In our synthetic benchmark, we shall consider Exponential and Gamma distributions, even though explicit formula may be found for the pair 
$(\ta,\vta)$.
\ \vspace{1ex} \\
First of all, we wish to point out that our recursive procedure is very fast for both algorithms since a set of $1000$ observations is handled 
in less than 0.1 second with a standard laptop. Next, Figure \ref{fig:asconvergence} illustrates the good almost sure behavior of the 
standard and convexified algorithms both on Exponential and Gamma distributions. Here, we consider
the $\cE(1/10)$ and $\cG(4,3)$ distributions. 
\begin{figure}[h]
\begin{minipage}[b]{0.4\linewidth}
\centerline{\includegraphics[width=8cm]{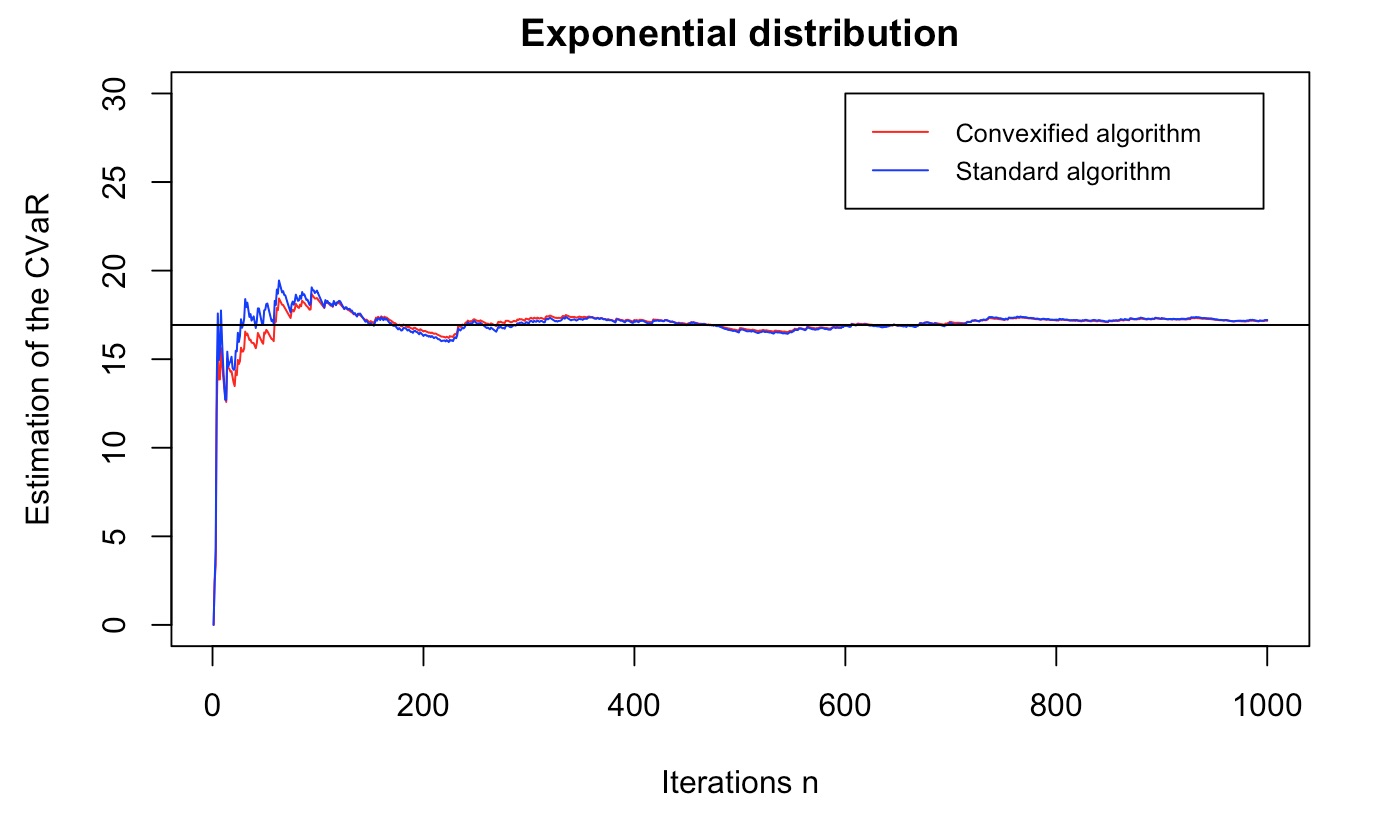}}
\end{minipage}
\hspace{2cm}
\begin{minipage}[b]{0.4\linewidth}
\centerline{\includegraphics[width=8cm]{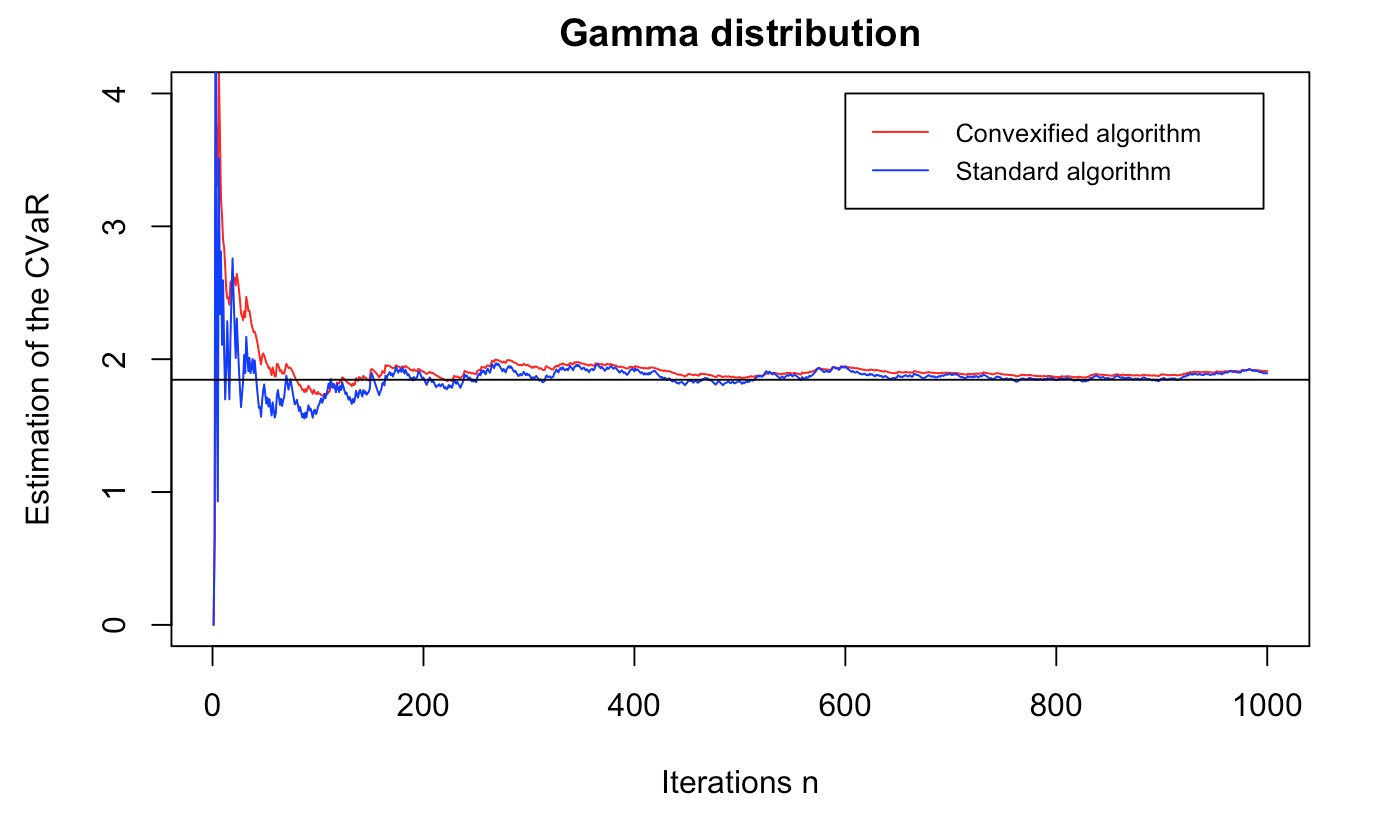}}
\end{minipage}
\vspace{-1em}
\caption{Almost sure convergence of our algorithms for $\alpha=0.5$ and $b_n=1/n$.
\label{fig:asconvergence}}
\end{figure}

Second, one can verify and compare the limiting variance of the asymptotic normality involved in 
Theorem \ref{T-AN} for several values of $a$ and $b$. 
Figure \ref{fig:CLT} represents the histogram of the rescaled algorithms for several values of $a$ and $b$.
One can check that the convexified algorithm outperforms the standard algorithm as soon as $b<a$.

%\vspace{-1em}
\begin{figure}[h]
\begin{minipage}[b]{0.4\linewidth}
\centerline{\includegraphics[width=6.5cm]{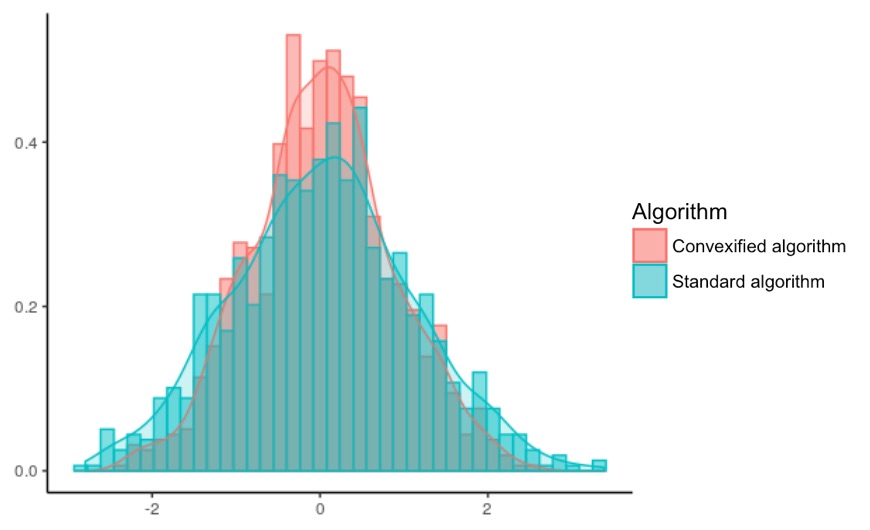}}
\end{minipage}\hfill
\begin{minipage}[b]{0.4\linewidth}
\centerline{\includegraphics[width=8cm]{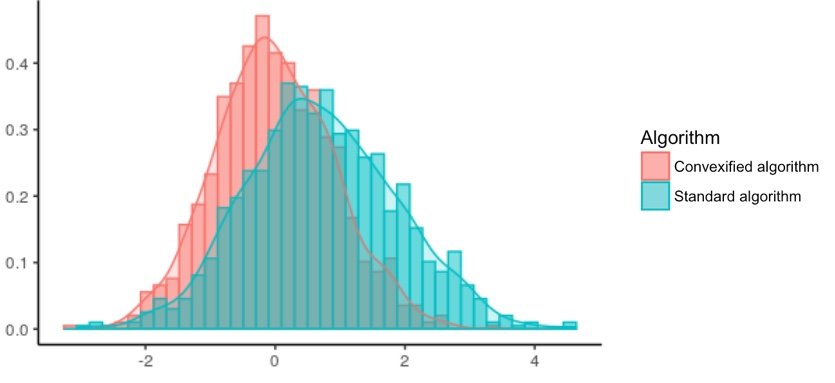}}
\end{minipage}\\
\begin{minipage}[b]{0.4\linewidth}
\centerline{\includegraphics[width=8cm]{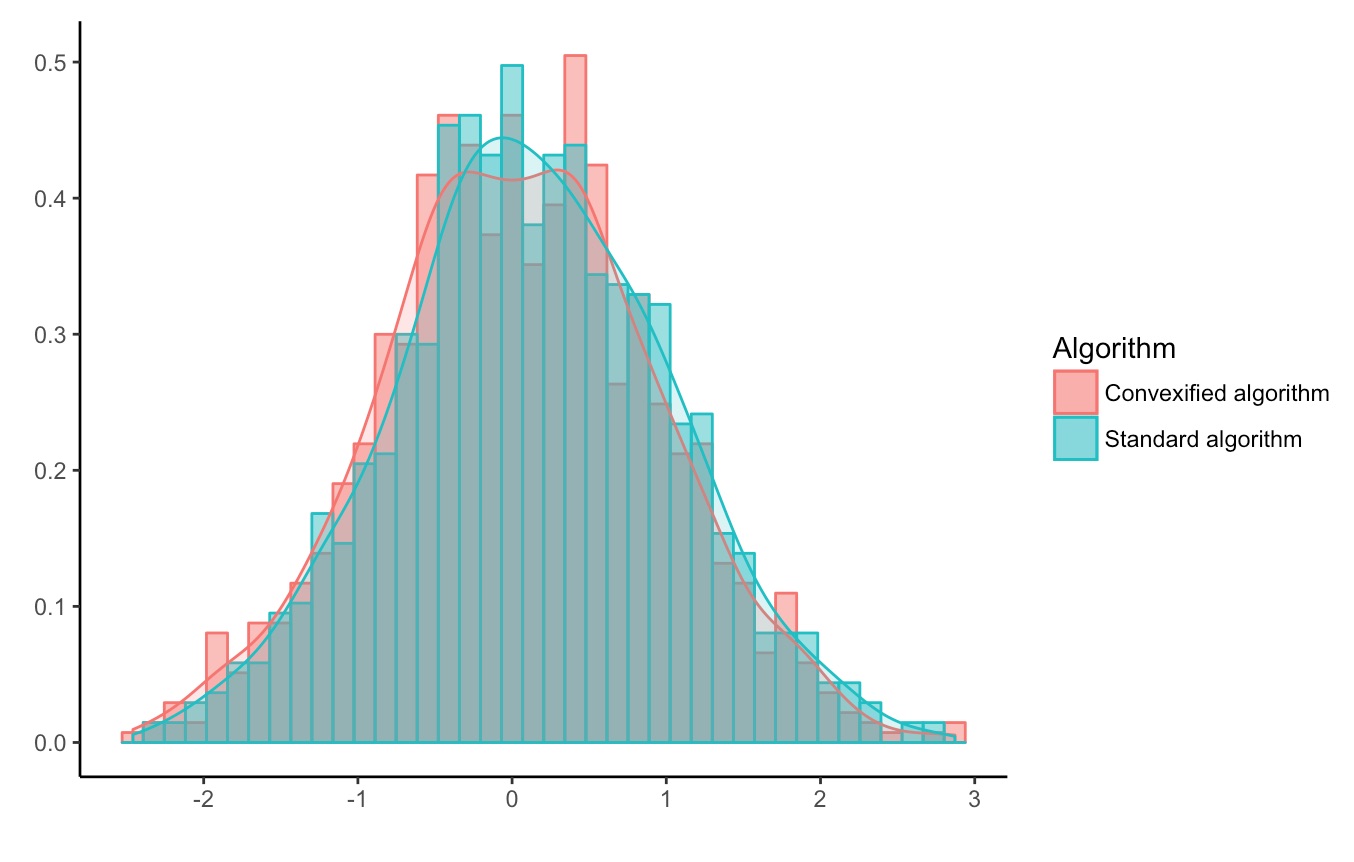}}
\end{minipage}\\
\vspace{-1em}
\caption{ Distribution of the rescaled algorithms in different situations: top-left ($a=2/3<b=4/5<1$), top-right  ($b=2/3<a=4/5<1$), bottom ($a=2/3<b=1$). 
One can verify the asymptotic normality with larger variance for the standard rescaled algorithm (top-right). \label{fig:CLT}}
\end{figure}

One can also use our method to estimate online $95\%$ confidence intervals for the superquantile $\vta$ 
as explained in Remark 3.4. This is illustrated in Figure \ref{fig:IC-online} with the Exponential and Gamma distributions 
with $a=2/3$ and $b=1$.

\begin{figure}[h]
\begin{minipage}[b]{0.4\linewidth}
\centerline{\includegraphics[width=7cm]{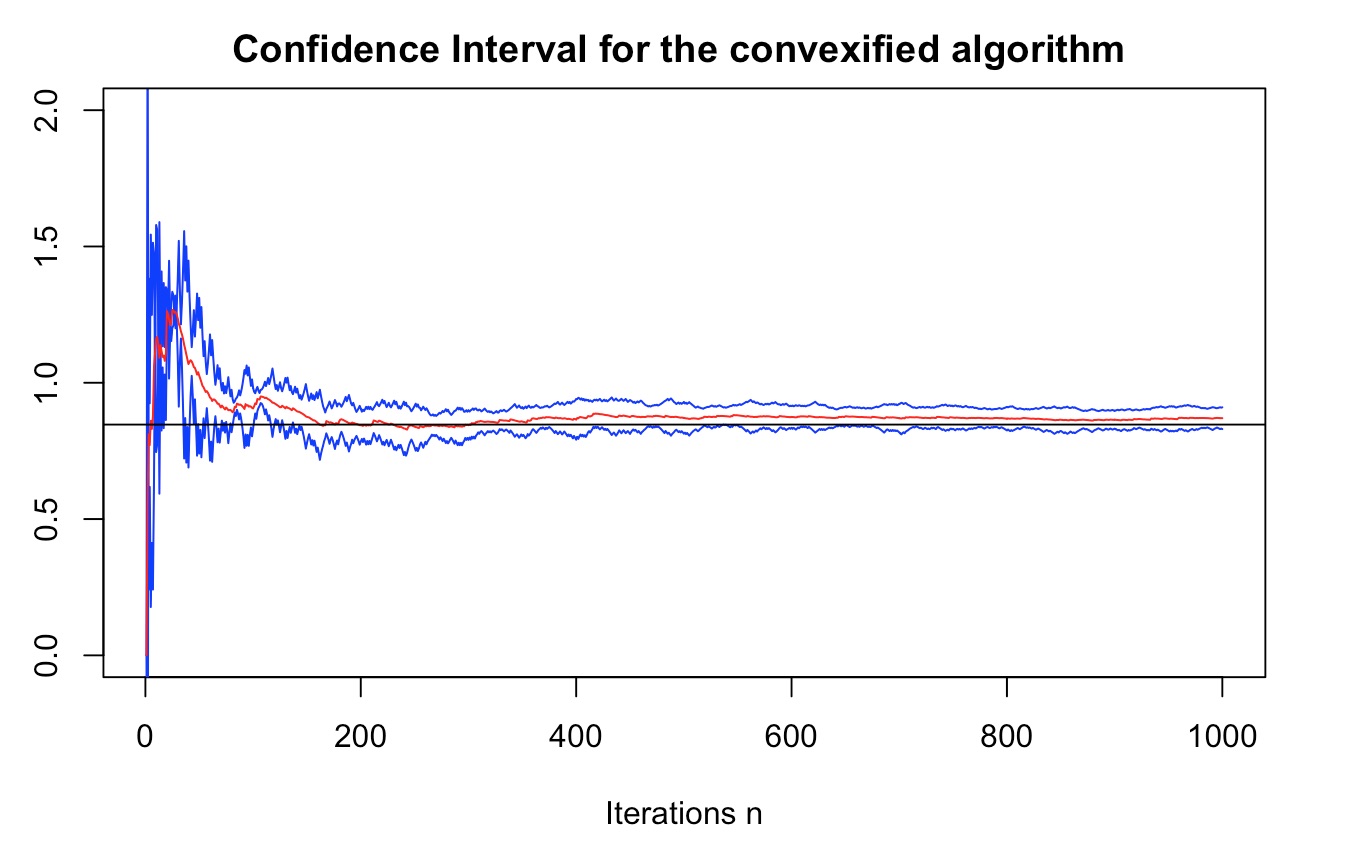}}
\end{minipage}
\hspace{2cm}
\begin{minipage}[b]{0.4\linewidth}
\centerline{\includegraphics[width=7cm]{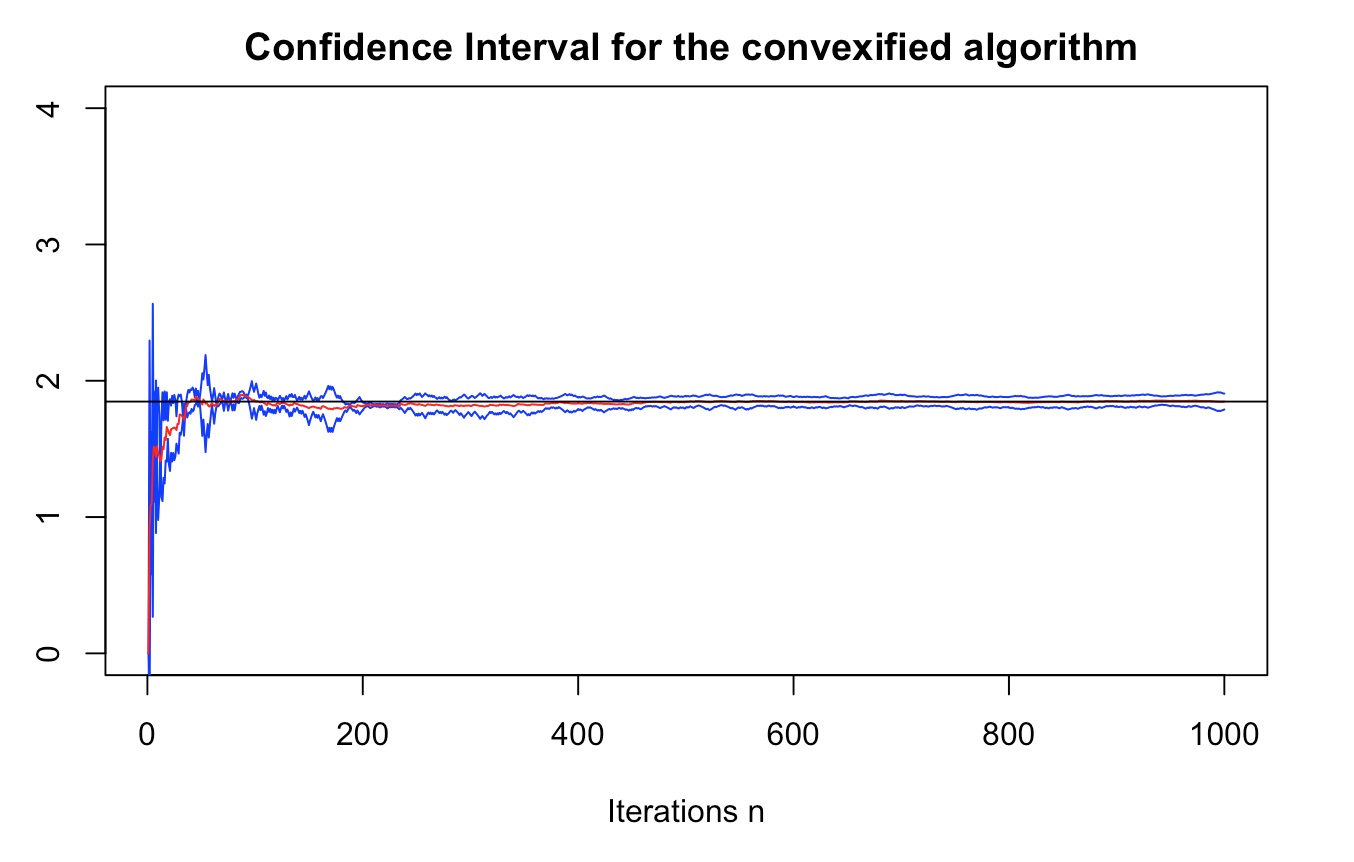}}
\end{minipage}
\vspace{-1em}
\caption{Online confidence interval with the convexified algorithm for the Exponential and Gamma distributions.}
\label{fig:IC-online}
\end{figure}

\subsection{Real data}
We finally illustrate, as a proof of concept, the use of our two algorithms
on financial real-data that are freely available on the R-package \textrm{tseries} (EuStockMarkets dataset). 
Some more recent ressources may also be downloaded  on the Yahoo! Finance website.
We consider the four time series of the financial stock-markets \textrm{DAX, CAC40, SMI, FTSE} between 2014 and 2018 and compute 
the CVaR of the weekly log-returns, that are common indicators in the analysis of financial markets.
It is commonly admitted as a reasonnable approximation that in non-exceptionnal situations, the log-returns are not far 
from an independent and identically distributed set of observations. As a major interest in finance, we compute the 
negative CVaR at the level $10\%$ and some $95\%$ confidence intervals as well. Our results are presented in 
Figure \ref{fig:finance} for the convexified algorithm tuned with the parameters $a=2/3$, $a_1=5$ and $b=1$, $b_1=3/4$.

%\vspace{-1em}
\begin{figure}[h]
\begin{minipage}[b]{0.4\linewidth}
\centerline{\includegraphics[height=3cm,width=10cm]{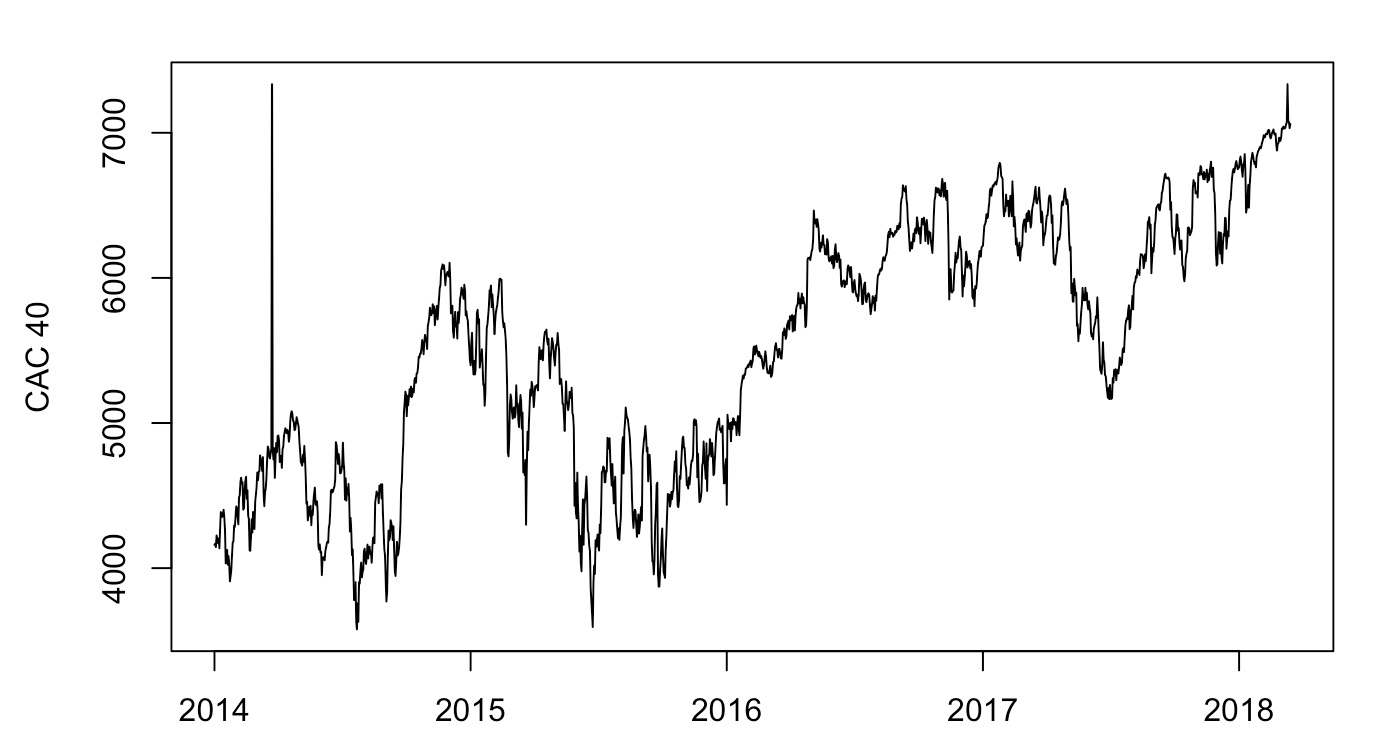}}\centerline{\includegraphics[height=3cm,width=10cm]{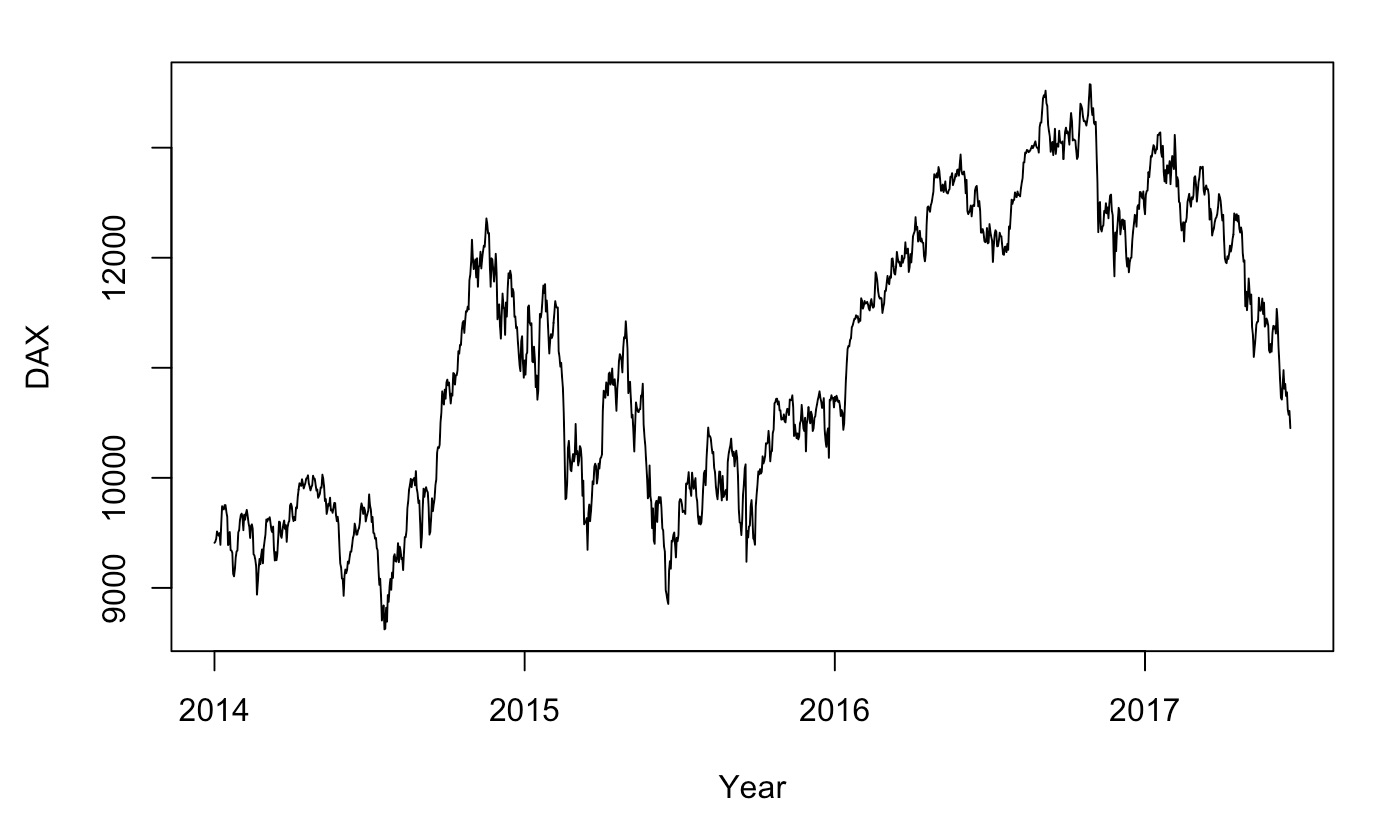}}
\end{minipage}\hfill \\
\begin{minipage}[b]{0.4\linewidth}
\centerline{\includegraphics[width=6.5cm]{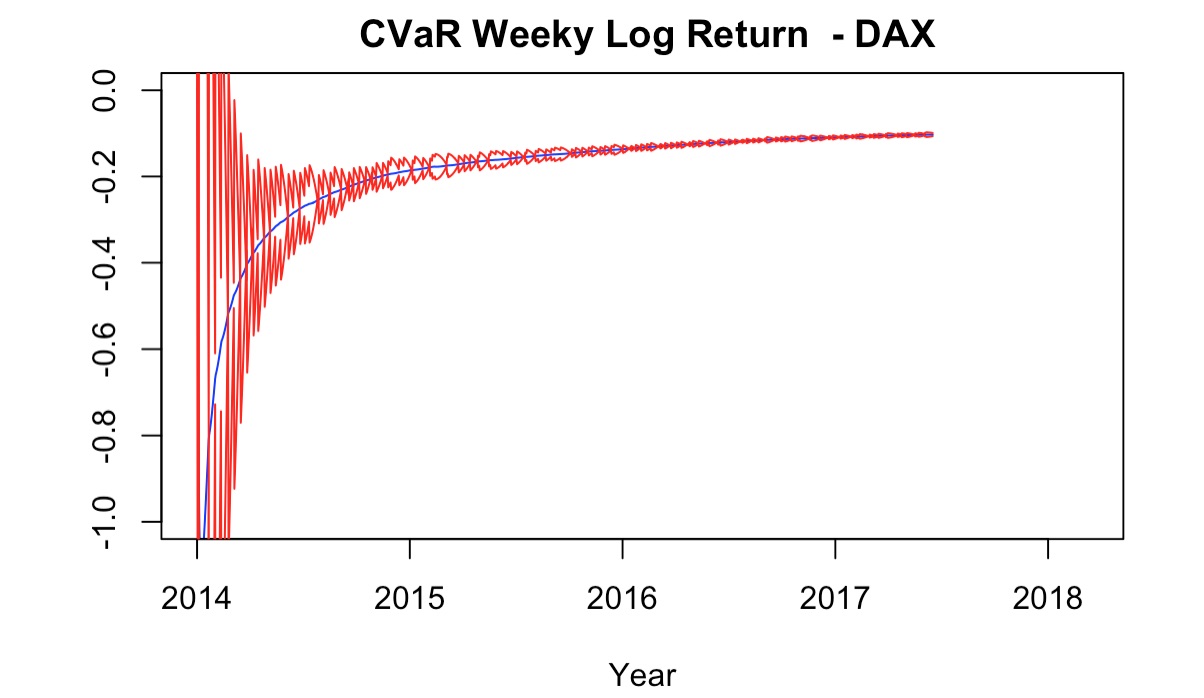}}
\end{minipage}
\begin{minipage}[b]{0.4\linewidth}
\centerline{\includegraphics[width=6.5cm]{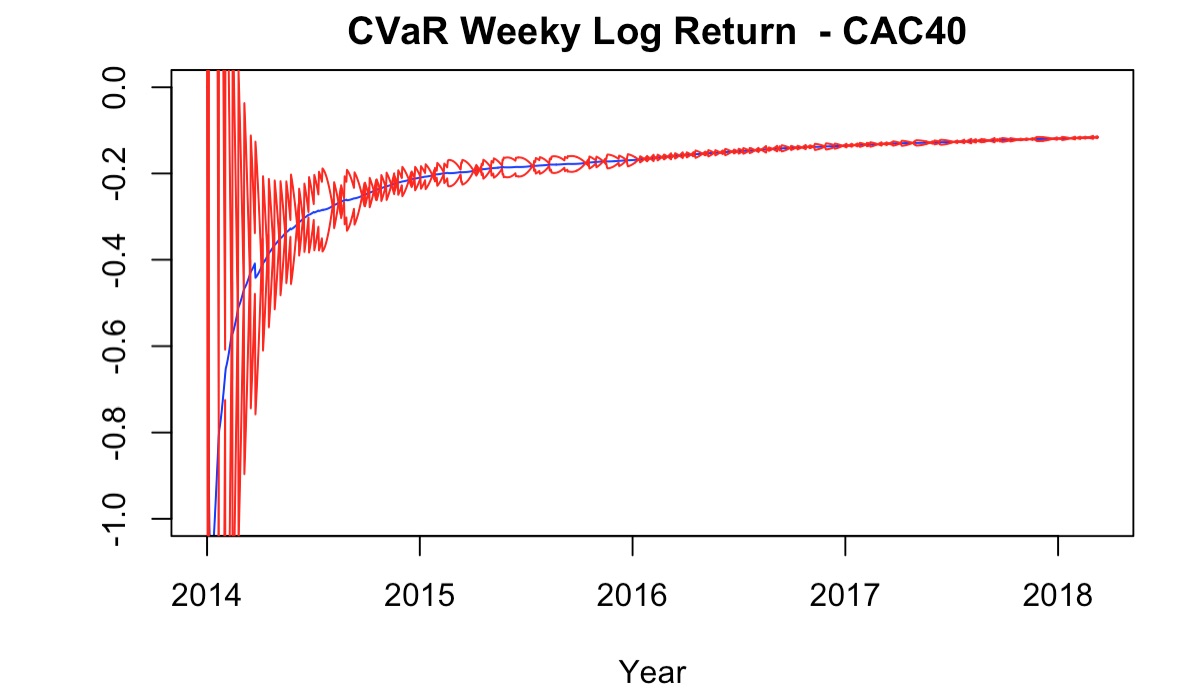}}
\end{minipage}\\
\vspace{-1em}
\caption{Convexified algorithm on Yahoo! Finance datasets.\label{fig:finance}}
\end{figure}

%\bibliographystyle{abbrv}
%\bibliography{BCG-SUPERQ-2020}  

\begin{thebibliography}{10}

\bibitem{Bahadur1966}
R.~R. Bahadur.
\newblock A note on quantiles in large samples.
\newblock {\em Ann. Math. Statist.}, 37:577--580, 1966.

\bibitem{Bardou2009}
O.~Bardou, N.~Frikha, and G.~Pages.
\newblock Computing var and cvar using stochastic approximation and adaptive
  unconstrained importance sampling.
\newblock {\em Monte Carlo Methods and Applications}, 15(3):173--210, 2009.

\bibitem{Ben-Tal}
A.~Ben-Tal and M.~Teboulle.
\newblock Expected utility, penalty ffunctions, and duality in stochastic
  nonlinear programming.
\newblock {\em Management Science}, 32(11):1445--1466, 1986.

\bibitem{bercu2004}
B.~Bercu.
\newblock On the convergence of moments in the almost sure central limit
  theorem for martingales with statistical applications.
\newblock {\em Stochastic Process. Appl.}, 11:157--173, 2004.

\bibitem{Borkar1997}
V.~S. Borkar.
\newblock Stochastic approximation with two time scales.
\newblock {\em Systems Control Lett.}, 29:291--294, 1997.

\bibitem{CCG2017}
H.~Cardot, P.~C\'enac, and A.~Godichon-Baggioni.
\newblock Online estimation of the geometric median in {H}ilbert spaces:
  {N}onasymptotic confidence balls.
\newblock {\em Ann. Statist.}, 45(2):591--614, 2017.

\bibitem{CCZ2013}
H.~Cardot, P.~Cenac, and P.~A. Zitt.
\newblock Efficient and fast estimation of the geometric median in hilbert
  spaces with an averaged stochastic gradient algorithm.
\newblock {\em Bernoulli}, 19:18--43, 2013.

\bibitem{Chung1954}
K.~L. Chung.
\newblock On a stochastic approximation method.
\newblock {\em Ann. Math. Statist}, 25:463--483, 1954.

\bibitem{Duflo1997}
M.~Duflo.
\newblock {\em {Random iterative models}}, volume~34 of {\em {Applications of
  Mathematics}}.
\newblock Springer-Verlag, Berlin, 1997.

\bibitem{Fabian1968}
V.~Fabian.
\newblock On asymptotic normality in stochastic approximation.
\newblock {\em Ann. Math. Statist}, 39:1327--1332, 1968.

\bibitem{GPS2018}
S.~Gadat, F.~Panloup, and S.~Saadane.
\newblock Stochastic heavy ball.
\newblock {\em Electronic Journal of Statistics}, pages 461--529, 2018.

\bibitem{Gaposkin1975}
V.~Gaposkin and T.~Krasulina.
\newblock On the law of the iterated logarithm in stochastic approximation
  processes.
\newblock {\em Theory of Probability and its Applications}, 19(4):844--850,
  1975.

\bibitem{Ghosh1971}
J.~K. Ghosh.
\newblock A new proof of the {B}ahadur representation of quantiles and an
  application.
\newblock {\em Ann. Math. Statist.}, 42:1957--1961, 1971.

\bibitem{Godichon2015}
A.~Godichon-Baggioni.
\newblock Estimating the geometric median in hilbert spaces with stochastic
  gradient algorithms : Lp and almost sure rates of convergence.
\newblock {\em Journal of Multivariate Analysis}, pages 209--222, 2015.

\bibitem{Godichon2019}
A.~Godichon-Baggioni.
\newblock Online estimation of the asymptotic variance for averaged stochastic
  gradient algorithms,.
\newblock {\em J. Statist. Plann. Inference}, pages 1--19, 2019.

\bibitem{Kersting1977}
G.~Kersting.
\newblock Almost sure approximation of the robbins-monro process by sums of
  independent random variables.
\newblock {\em Ann. Probab.}, 5(6):954--965, 1977.

\bibitem{Konda2004}
V.~Konda and J.~N. Tsitsiklis.
\newblock Convergence rate of linear two-time-scale stochastic approximation.
\newblock {\em Ann. Appl. Probab.}, 14:796--819, 2004.

\bibitem{KushnerYin2003}
H.~J. Kushner and G.~G. Yin.
\newblock {\em Stochastic approximation and recursive algorithms and
  applications}, volume~35 of {\em Applications of Mathematics}.
\newblock Springer-Verlag, New York, second edition, 2003.
\newblock Stochastic Modelling and Applied Probability.

\bibitem{LaiRobbins1979}
T.~L. Lai and H.~Robbins.
\newblock Adaptive design and stochastic approximation.
\newblock {\em Ann. Statist.}, 7(6):1196--1221, 1979.

\bibitem{MokkademPelletier2006}
A.~Mokkadem and M.~Pelletier.
\newblock Convergence rate and averaging of nonlinear two-time-scale stochastic
  approximation algorithms.
\newblock {\em Ann. Appl. Probab.}, 16:1671--1702, 2006.

\bibitem{Pelletier1998}
M.~Pelletier.
\newblock On the almost sure asymptotic behaviour of stochastic algorithms.
\newblock {\em Stochastic Process. Appl. 78}, 2:217--244, 1998.

\bibitem{PolyakJuditsky1992}
B.~T. Polyak and A.~Juditsky.
\newblock Acceleration of stochastic approximation by averaging.
\newblock {\em SIAM Journal on Control and Optimization}, 30:838--855, 1992.

\bibitem{RobbinsMonro1951}
H.~Robbins and S.~Monro.
\newblock A stochastic approximation method.
\newblock {\em The Annals of Mathematical Statistics}, 22:400--407, 1951.

\bibitem{RobbinsSiegmund1971}
H.~Robbins and D.~Siegmund.
\newblock A convergence theorem for non negative almost supermartingales and
  some applications.
\newblock {\em Optimizing methods in stat.}, pages 233--257, 1971.

\bibitem{Rockafellar2000}
R.~T. Rockafellar and S.~Uryasev.
\newblock Optimization of conditional value-at-risk.
\newblock {\em The Journal of Risk}, 2(3):21--41, 2000.

\bibitem{Rockafellar2002}
R.~T. Rockafellar and S.~Uryasev.
\newblock Conditional value-at-risk for general loss distributions.
\newblock {\em Journal of Banking and Finance}, 26(7):1443--1471, 2002.

\bibitem{Ruppert1988}
D.~Ruppert.
\newblock Efficient estimations from a slowly convergent robbins-monro process.
\newblock {\em Technical Report, 781, Cornell university operations research
  and industrial engineering, Ithaca, NY}, 1988.

\bibitem{Sacks1958}
J.~Sacks.
\newblock Asymptotic distribution of stochastic spproximation procedures.
\newblock {\em The Annals of Mathematical Statistics}, 29:373--405, 1958.

\bibitem{Stout1970}
W.~Stout.
\newblock A martingale analogue of kolmogorov's law of the iterated logarithm.
\newblock {\em Z. Wahrscheinlichkeitstheorie verw. Geb.}, 15:279--290, 1970.

\end{thebibliography}
 
\providecommand{\AC}{A.-C}\providecommand{\CA}{C.-A}\providecommand{\CH}{C.-H}\providecommand{\CJ}{C.-J}\providecommand{\JC}{J.-C}\providecommand{\JP}{J.-P}\providecommand{\JB}{J.-B}\providecommand{\JF}{J.-F}\providecommand{\JJ}{J.-J}\providecommand{\JM}{J.-M}\providecommand{\KW}{K.-W}\providecommand{\PL}{P.-L}\providecommand{\RE}{R.-E}\providecommand{\SJ}{S.-J}\providecommand{\XR}{X.-R}\providecommand{\WX}{W.-X}

\end{document}